\documentclass[draft,12pt]{amsart}
\usepackage{amssymb,times}

\makeatletter
\def\@strippedMR{}
\def\@scanforMR#1#2#3\endscan{
  \ifx#1M\ifx#2R\def\@strippedMR{#3}
  \else\def\@strippedMR{#1#2#3}
  \fi\fi}
\renewcommand\MR[1]{\relax\ifhmode\unskip\spacefactor3000 \space\fi
  \@scanforMR#1\endscan
  MR\MRhref{\@strippedMR}{\@strippedMR}}
\makeatother

\newtheorem*{Thm*}{Theorem}
\newtheorem{Thm}{Theorem}
\newtheorem{Cor}[Thm]{Corollary}
\newtheorem{Prop}[Thm]{Proposition}
\newtheorem{Lemma}{Lemma}

\theoremstyle{definition}
\newtheorem{Defn}{Definition}

\newtheorem{Remark}{Remark}
\newtheorem{Ex}[Remark]{Example}

\newcommand{\mf}[1]{\mathbb{#1}}
\newcommand{\mc}[1]{\mathcal{#1}}
\newcommand{\mb}[1]{\mathbf{#1}}
\newcommand{\mk}[1]{\mathfrak{#1}}

\DeclareMathOperator{\NC}{\mathit{NC}}
\DeclareMathOperator{\Falg}{\mathcal{F}_{\mathrm{alg}}}
\DeclareMathOperator{\Part}{\mathcal{P}}
\DeclareMathOperator{\Int}{\mathit{Int}}

\newcommand{\norm}[1]{\left\Vert#1\right\Vert}
\newcommand{\abs}[1]{\left\vert#1\right\vert}
\newcommand{\chf}[1]{\mathbf{1}_{#1}}
\newcommand{\set}[1]{\left\{#1\right\}}
\newcommand{\ip}[2]{\left \langle #1, #2 \right \rangle}
\newcommand{\state}[1]{\varphi \left[ #1 \right]}
\newcommand{\State}[1]{\Phi \left[ #1 \right]}
\renewcommand{\phi}{\varphi}

\newcommand{\FreeCum}[1]{R \left[ #1 \right]}

\newcommand{\Cum}[1]{\eta \left[ #1 \right]}
\newcommand{\Alg}[1]{\mathit{Alg} \left( #1 \right)}

\newcommand{\br}{\medskip\noindent}

\allowdisplaybreaks[1]

\title[Appell polynomials II. Boolean theory]{Appell polynomials and their relatives II. Boolean theory}
\author[M.~Anshelevich]{Michael Anshelevich}
\thanks{This work was supported in part by NSF grant DMS-0613195}
\address{Department of Mathematics, Texas A\&M University, College Station, TX 77843-3368}
\email{manshel@math.tamu.edu}
\date{\today}

\begin{document}

\begin{abstract}
The Appell-type polynomial family corresponding to the simplest non-commutative derivative operator turns out to be connected with the Boolean probability theory, the simplest of the three universal non-commutative probability theories (the other two being free and tensor/classical probability). The basic properties of the Boolean Appell polynomials are described. In particular, their generating function turns out to have a resolvent-type form, just like the generating function for the \emph{free} Sheffer polynomials. It follows that the Meixner (that is, Sheffer plus orthogonal) polynomial classes, in the Boolean and free theory, coincide. This is true even in the multivariate case. A number of applications of this fact are described, to the Belinschi-Nica and Bercovici-Pata maps, conditional freeness, and the Laha-Lukacs type characterization.

\br
A number of properties which hold for the Meixner class in the free and classical cases turn out to hold in general in the Boolean theory. Examples include the behavior of the Jacobi coefficients under convolution, the relationship between the Jacobi coefficients and cumulants, and an operator model for cumulants. Along the way, we obtain a multivariate version of the Stieltjes continued fraction expansion for the moment generating function of an arbitrary state with monic orthogonal polynomials.
\end{abstract}


\maketitle

\section{Introduction}

\br
In 1880, Paul Appell \cite{Appell} investigated families of polynomials $\set{A_n(x)}$ with the property that
\[
A_n'(x) = n A_{n-1}(x).
\]
This recursion determines each $A_n$ except for the constant term. One way to fix the constant terms \cite{Thorne-Appell-sets} is to require that
\begin{equation}
\label{Centered}
\mu(A_n) = \int A_n(x) \,d\mu(x) = \delta_{n, 0},
\end{equation}
where $\mu$ is a probability measure (a positive measure of total integral $1$), or a (not necessarily positive) linear functional on polynomials such that $\mu(1) = 1$. The three most familiar examples of Appell polynomials are
\begin{enumerate}
\item
$A_n(x) = x^n$, corresponding to delta-measure $\mu = \delta_0$,
\item
$A_n(x) = H_n(x)$, the Hermite polynomials, corresponding to $d\mu(x) = \frac{1}{\sqrt{2 \pi}} e^{-x^2/2} \,dx$, and
\item
$A_n(x) = B_n(x)$, the Bernoulli polynomials, corresponding to $d \mu(x) = \chf{[0,1]} \,dx$.
\end{enumerate}

\br
Besides the importance of these and other specific examples, general Appell polynomials have nice combinatorial properties. They also play a role in probability theory: it is easy to see that the Appell polynomials are martingales for the corresponding L{\'e}vy processes \cite{SchOrthogonal}, but there are other more surprising appearances of these polynomials, such as \cite{Giraitis-Surgailis,Avram-Taqqu} or \cite{Kyprianou-Novikov-Shiryaev}. Here it is crucial that multivariate versions of the Appell polynomials exist (instead of the derivative operator, one starts with a $d$-tuple of partial derivatives), and even a family of more ``functorial'' Appell maps, which in this context are sometimes called \emph{Wick products}.

\br
An equivalent definition of the Appell polynomials corresponding to a functional $\mu$ is via their generating function
\[
\sum_{n=0}^\infty \frac{1}{n!} A_n(x) z^n = \exp \left( x z - \ell(z) \right),
\]
where
\[
\ell(z) = \log \mu \left( e^{x z} \right)
\]
is the logarithm of the exponential moment generating function of $\mu$, and is the generating function for \emph{cumulants} of $\mu$. Allowing for general invertible changes of variable $u(z)$ gives a larger class of Sheffer polynomials, with generating functions of the form
\[
\sum_{n=0}^\infty \frac{1}{n!} P_n(x) z^n = \exp \left( x u(z) - \ell(u(z)) \right),
\]
which are the main objects in Rota's umbral calculus \cite{RotaFiniteCalculusBook,Roman-Book}. Restricting now to positive measures $\mu$, \emph{Meixner} polynomials are Sheffer polynomials which are also orthogonal, and Meixner distributions are the measures whose orthogonal polynomials are Meixner. These measures are known \cite{Meixner}, and consist of the normal, Poisson, Gamma, Binomial, negative binomial, and hyperbolic secant distributions. Again, much of the analysis can be done in the multivariate case. In that case, the description of the Meixner class is not complete, but there are partial classification results due to Letac, Casalic, Feinsilver, Pommeret, and others.

\br
In Part I \cite{AnsAppell}, I introduced the \emph{free Appell polynomials}. Here (in the single variable case) the starting point is the difference quotient operator
\[
(\partial f)(x,y) = \frac{f(x) - f(y)}{x - y},
\]
and the polynomials are defined recursively via
\[
(\partial A_n)(x, y) = \sum_{k=0}^{n-1} A_k(x) A_{n-k-1} (y)
\]
and the centering condition~\eqref{Centered}. Note that
\[
(\partial x^n) = \sum_{k=0}^{n-1} x^k y^{n-k-1},
\]
so $\set{x^n}$ are also free Appell. Chebyshev polynomials of the second kind are another example. These polynomials have been considered previously under the name of ``sequences of interpolatroy type'' \cite{Verde-Star}. Free Appell polynomials also turn out to have nice, similar but different, combinatorial properties, and probabilistic connections. This time, however, the connection is to \emph{free} probability theory. This is a non-commutative probability theory, initiated by Voiculescu in the early 1980s, whose main objects, rather than commuting random variables, are non-commuting operators (see the next section). The theory is by now quite broad and has a number of results which parallel theorems in the usual probability theory, as well as a number of applications and connections to other fields, notably to the theories of operator algebras and random matrices. The connection of free Appell polynomials to free probability theory is not unexpected, since the importance of the difference quotient was discovered by Voiculescu in his work on free entropy \cite{Voi-Entropy5,VoiCoalgebra}.

\br
Again, one can define multivariate free Appell polynomials, except in this case it is natural to consider polynomials in multiple \emph{non-commuting} variables. The extension of the class to free Sheffer, and the restriction to the free Meixner class was carried out in \cite{AnsMulti-Sheffer,AnsFree-Meixner}, and some of these results are summarized in Section~\ref{Section:Meinxer}. They include conditions on the free cumulant generating function of a free Meixner state $\phi$ (a multivariate version of the measure $\mu$) and an operator model for it.

\br
In this paper, we consider the program described above where the starting point is yet another derivative operator
\[
(D f)(x) = \frac{f(x) - f(0)}{x}.
\]
In other words, we define the \emph{Boolean} Appell polynomials by the recursion
\[
D A_n(x) = A_{n-1}(x)
\]
(note that $D x^n = x^{n-1}$) and condition~\eqref{Centered}. $D$ is the $q=0$ version of the $q$-derivative
\[
D_q f(x) = \frac{f(x) - f(q x)}{(1-q) x},
\]
and so the Boolean Appell polynomials in a single variable are a particular case of the $q$-Appell polynomials of \cite{Al-Salam-q-Appell}.

\br
The probabilistic connection of the Boolean Appell polynomials is to a different non-commutative probability theory, the so-called Boolean probability. The first examples from this theory date from the 1970s \cite{vWa73,Bozejko-Riesz-product,Bozejko-Free-groups}. The Boolean theory is much simpler than free probability theory, and at this point lacks its depth, primarily because of the lack of the random matrix techniques. Nevertheless, it has a number of crucial properties since it comes from a universal product; in fact, by a theorem of Speicher~\cite{SpeUniv} (see also \cite{Ben-Ghorbal-Independence,Muraki-Natural-products}), Boolean, free, and classical theories are exactly the only ones which arise from universal products (respectively, Boolean, free, and tensor) which do not depend on the order of the components. A sample of work on Boolean probability theory includes \cite{SW97,Privault-Boolean,Oravecz-Fermi,Franz-Unification,Krystek-Yoshida-t,Mlotkowki-Lambda-Boolean,
Lenczewski-Noncommutative-independence,Stoica-Boolean,Bercovici-Boolean}.

\br
We start the paper by defining multivariate Boolean Appell polynomials and describing their generating functions, recursion relations, and other basic properties, which are analogous to the classical and the free case. Then we use the generating function form to define the Boolean Sheffer class. Here the analogy breaks down and a new phenomenon appears. Namely, while the free and Boolean Appell classes differ, the free and Boolean Sheffer classes happen to coincide, even the the multivariate case. It follows that the free and Boolean Meixner classes (Sheffer $+$ Orthogonal) coincide as well. So for a free / Boolean Meixner state $\phi$, in addition to conditions on its free cumulant generating function, we get similar conditions on its Boolean cumulant generating function. One consequence of these relations is that the free / Boolean Meixner class is preserved under the transformation $\mf{B}_t$ from \cite{Belinschi-Nica-B_t,Belinschi-Nica-Free-BM}. Another is the explanation of some limit theorems from the theory of conditional freeness. A third is the observation that the Boolean-to-free version of the Bercovici-Pata bijection \cite{BerPatDomains,Belinschi-Nica-Eta} takes the free Meixner class to itself, and has a simple explicit form on it.

\br
Laha and Lukacs \cite{Laha-Lukacs} characterized all the (classical) Meixner distributions using a quadratic regression property, and Bo{\.z}ejko and Bryc \cite{Boz-Bryc} proved the corresponding free version. We prove a Boolean version of this property, which however holds not for the full free / Boolean Meixner class, but only for a smaller class of Bernoulli distributions.

\br
Another place where the analogy between classical and free theories on one hand, and Boolean on the other, breaks down is in the relation between cumulants, convolution, and the Jacobi parameters. If $\set{\mu_t}$  is a convolution semigroup of Meixner distributions, and $\set{P_n(x,t)}$ are the corresponding orthogonal polynomials, they satisfy recursion relations of the type
\[
x P_n(x,t) = P_{n+1}(x,t) + (t + n b) P_n(x,t) + n (t + (n-1)c) P_{n-1}(x,t).
\]
Similarly, if instead $\set{\mu_t}$ are a free convolution semigroup of free Meixner distributions, the corresponding orthogonal polynomials satisfy
\begin{align*}
x P_0(x,t) & = P_1(x,t) + t P_0(x,t), \\
x P_1(x,t) & = P_2(x,t) + (t + b) P_1(x,t) + t P_0(x,t), \\
x P_n(x,t) & = P_{n+1}(x,t) + (t + b) P_n(x,t) + (t + c) P_{n-1}(x,t).
\end{align*}
In both cases, the Jacobi coefficients are linear in $t$. This is certainly not the case for general (classical or free) convolution semigroups. Similarly, there is not in general a simple relation between the (classical or free) cumulants of a measure and the Jacobi parameters of its orthogonal polynomials.

\br
In contrast, in the Boolean case these special properties of the Meixner class actually hold for all measures: Jacobi coefficients are linear in the Boolean convolution parameter, and Boolean cumulants have a simple expression in terms of the Jacobi parameters. The single-variable versions of these statements are known \cite{Boz-Wys,Lehner-Cumulants-lattice}, and we show that they hold for arbitrary states which have monic orthogonal polynomials. Along the way, we construct a multivariate continued fraction expansion for a moment generating function of any such state, a result which is of independent interest.

\br
The paper is organized as follows. Section~\ref{Section:Preliminaries} sets out the notation and background results. Section~\ref{Section:Appell} describes the Boolean Appell polynomials. Section~\ref{Section:Meinxer} describes the coincidence of the Boolean and free Meixner classes. Finally, the general results about the continued fraction expansion and the operator representation of Boolean cumulants are contained in the appendix.

\br
\textbf{Acknowledgements.}
I thank Andu Nica, Serban Belinschi, and Wlodek Bryc for useful and enjoyable conversations.

\section{Preliminaries}
\label{Section:Preliminaries}

\br
Variables in this paper will typically come in $d$-tuples, which will be denoted using the bold font: $\mb{x} = (x_1, x_2, \ldots, x_d)$, and the same for $\mb{z}, \mb{S}$, etc.

\subsection{Operations on power series}

Let $\mb{z} = (z_1, \ldots, z_d)$ be non-commuting indeterminates. For a non-commutative power series $G$ in $\mb{z}$ and $i = 1, \ldots, d$, define the left non-commutative partial derivative $D_i G$ by a linear extension of $D_i(1) = 0$,
\[
D_i z_{\vec{u}} = \delta_{i u(1)} z_{u(2)} \ldots z_{u(n)}.
\]
Denote by $\mb{D} G = (D_1 G, \ldots, D_d G)$ the left non-commutative gradient.

\br
For a non-commutative power series $G$, denote by $G^{-1}$ its inverse with respect to multiplication. For a $d$-tuple of non-commutative power series $\mb{G} = (G_1, \ldots, G_d)$, denote by $\mb{G}^{\langle -1 \rangle}$ its inverse with respect to composition (which is also a $d$-tuple).

\subsection{Polynomials}

Let $\mf{C}\langle \mb{x} \rangle = \mf{C}\langle x_1, x_2, \ldots, x_d \rangle$ be all the polynomials with complex coefficients in $d$ non-commuting variables. \emph{Multi-indices} are elements $\vec{u} \in \set{1, \ldots, d}^k$ for $k \geq 0$; for $\abs{\vec{u}} = 0$ denote $\vec{u}$ by $\emptyset$. Monomials in non-commuting variables $(x_1, \ldots, x_d)$ are indexed by such multi-indices:
\[
x_{\vec{u}} = x_{u(1)} \ldots x_{u(k)}.
\]
Note that our use of the term ``multi-index'' is different from the usual one, which is more suited for indexing monomials in commuting variables.

\br
For two multi-indices $\vec{u}, \vec{v}$, denote by $(\vec{u}, \vec{v})$ their concatenation. For $\vec{u}$ with $\abs{\vec{u}} = k$, denote
\[
(\vec{u})^{op} = (u(k), \ldots, u(2), u(1)).
\]
Define an involution on $\mf{C}\langle \mb{x} \rangle$ via the linear extension of
\[
(\alpha x_{\vec{u}})^\ast = \bar{\alpha} x_{(\vec{u})^{op}},
\]
so that each $x_i$ is self-adjoint.

\br
A \emph{monic polynomial family} in $\mb{x}$ is a family $\set{P_{\vec{u}}(\mb{x})}$ indexed by all multi-indices
\[
\bigcup_{k=1}^\infty \set{\vec{u} \in \set{1, \ldots, d}^k}
\]
(with $P_{\emptyset} = 1$ being understood) such that\[
P_{\vec{u}}(\mb{x}) = x_{\vec{u}} + \textsl{lower-order terms}.
\]
Note that $P_{\vec{u}}^\ast \neq P_{(\vec{u})^{op}}$ in general.

\subsection{Algebras and states}

Let $\mc{A}$ be a complex $\ast$-algebra. Denote by
\[
\mc{A}^{sa} = \set{X \in \mc{A}: X = X^\ast}
\]
its self-adjoint part. A state $\Phi: \mc{A} \rightarrow \mf{C}$ is a linear functional which is unital, that is $\State{1} = 1$ if $\mc{A}$ has a unit, compatible with the $\ast$-operation so that for any $X \in \mc{A}$,
\[
\State{X^\ast} = \overline{\State{X}},
\]
and positive, that is for any $X \in \mc{A}$
\[
\State{X^\ast X} \geq 0.
\]
We will think of the pair $(\mc{A}, \Phi)$ as a non-commutative probability space, and refer to its elements as (non-commutative) random variables.

\br
A state $\Phi$ induces the pre-inner product
\[
\ip{X}{Y} = \State{X^\ast Y} = \overline{\ip{Y}{X}}
\]
and the seminorm
\[
\norm{X} = \sqrt{\State{X^\ast X}}.
\]
We will also denote
\[
\norm{X}_{\infty} = \sup_{Y \in \mc{A}} \frac{\norm{X Y}}{\norm{Y}} = \sup_{Y \in \mc{A}} \sqrt{\frac{\State{Y^\ast (X^\ast X) Y}}{\State{Y^\ast Y}}}.
\]

\br
Most of the time we will be working with states on $\mf{C} \langle \mb{x} \rangle$ arising as joint distributions. For
\[
X_1, X_2, \ldots, X_d \in \mc{A}^{sa},
\]
their joint distribution with respect to $\Phi$ is a state on $\mf{C} \langle \mb{x} \rangle$ determined by
\[
\state{P(x_1, x_2, \ldots, x_d)} = \State{P(X_1, X_2, \ldots, X_d)}
\]
The numbers $\state{x_{\vec{u}}}$ are called the \emph{moments} of $\phi$. More generally, for $d$ non-commuting indeterminates $\mb{z} = (z_1, \ldots, z_d)$, the series
\[
M(\mb{z}) = \sum_{\vec{u}} \state{x_{\vec{u}}} z_{\vec{u}}
\]
is the moment generating function of $\phi$. In the remainder of the paper, except for the appendix, we will assume that under the state $\phi$, the variables have zero mean and identity covariance,
\[
\state{x_i} = 0, \qquad \state{x_i x_j} = \delta_{ij}.
\]

\br
A state $\phi$ on $\mf{C} \langle \mb{x} \rangle$ has a \emph{monic orthogonal polynomial system}, or MOPS, if for any multi-index $\vec{u}$, there is a monic polynomial $P_{\vec{u}}$ with leading term $x_{\vec{u}}$, such that these polynomials are orthogonal with respect to $\phi$, that is,
\[
\ip{P_{\vec{u}}}{P_{\vec{v}}}_\phi = 0
\]
for $\vec{u} \neq \vec{v}$.

\subsection{Partitions}

A \emph{partition} $\pi$ of a set $V \subset \mf{Z}$ is a collection of disjoint subsets of $V$ (classes of $\pi$), $\pi = (B_1, B_2, \ldots, B_k)$, whose union equals $V$. We denote the collection of all partitions by $\Part(V)$. Most of the time we will be interested in partitions $\Part(n)$ of $\set{1, 2, \ldots, n}$. Partitions form a partially ordered set (in fact a lattice) under the operation of refinement, so that the largest partition is $\hat{1} = \bigl( \set{1, 2, \ldots, n} \bigr)$ and the smallest partition is $\hat{0} = \bigl( \set{1}, \set{2}, \ldots, \set{n} \bigr)$. We will use $i \stackrel{\pi}{\sim} j$ to denote that $i, j$ lie in the same class of $\pi$.

\br
Let $\NC(V)$ denote the collection of non-crossing partitions of $V$, which are partitions $\pi$ such that
\[
i \stackrel{\pi}{\sim} i', j \stackrel{\pi}{\sim} j', i \stackrel{\pi}{\not \sim} j, i < j < i' \Rightarrow i < j' < i'.
\]

\br
Let $\Int(V)$ denote the collection of interval partitions of $V$, which are partitions $\pi$ such that
\[
i \stackrel{\pi}{\sim} i',  i < j < i' \Rightarrow i \stackrel{\pi}{\sim} j \stackrel{\pi}{\sim} i'.
\]
For each $n$, let $\NC(n)$ (resp., $\Int(n)$) denote the collection of non-crossing (interval) partitions of the set $\set{1, 2, \ldots, n}$. Both $\NC(V)$ and $\Int(V)$ are sub-lattices of the lattice of all partitions; in fact, they have an additional property of being self-dual. As a lattice, $\Int(n)$ is isomorphic to the lattice of all subsets of a set of $(n-1)$ elements, and for this reason interval partitions are sometimes called Boolean partitions.

\subsection{Cumulants}

The Boolean (respectively, free) cumulant functional $\eta$ (resp., $R$) corresponding to a state $\phi$ is the linear functional on $\mf{C} \langle \mb{x} \rangle$ defined recursively by $\FreeCum{1} = \Cum{1} = 0$ and for $\abs{\vec{u}} = n$,
\begin{equation}
\label{Cumulants-definition}
\Cum{x_{\vec{u}}}
= \state{x_{\vec{u}}} - \sum_{\substack{\pi \in \Int(n), \\ \pi \neq \hat{1}}} \prod_{B \in \pi} \Cum{\prod_{i \in B} x_{u(i)}},
\end{equation}
resp.
\begin{equation}
\label{Free-cumulants-definition}
\FreeCum{x_{\vec{u}}}
= \state{x_{\vec{u}}} - \sum_{\substack{\pi \in \NC(n), \\ \pi \neq \hat{1}}} \prod_{B \in \pi} \FreeCum{\prod_{i \in B} x_{u(i)}},
\end{equation}
which expresses $\Cum{x_{\vec{u}}}$ and $\FreeCum{x_{\vec{u}}}$ in terms of the joint moments and sums of products of lower-order cumulants. From these, we can form the Boolean (resp., free) cumulant generating function $\eta$ (resp., $R$) of $\phi$ via
\begin{equation*}
\eta(z_1, z_2, \ldots, z_d)
= \sum_{n=1}^\infty \sum_{\abs{\vec{u}} = n} \Cum{x_{\vec{u}}} z_{\vec{u}},
\end{equation*}
resp.
\begin{equation*}
R(z_1, z_2, \ldots, z_d)
= \sum_{n=1}^\infty \sum_{\abs{\vec{u}} = n} \FreeCum{x_{\vec{u}}} z_{\vec{u}},
\end{equation*}
where $\mb{z} = (z_1, \ldots, z_d)$ are non-commuting indeterminates. If $M(\mb{z})$ is the moment generating function of $\phi$, then from definitions~\eqref{Cumulants-definition}, \eqref{Free-cumulants-definition} there follow the generating function relations
\begin{equation*}
\eta(\mb{z}) (1 + M(\mb{z})) = M(\mb{z}),
\end{equation*}
which is equivalent to
\begin{equation}
\label{K-M-formula}
\eta(\mb{z}) = 1 - (1 + M(\mb{z}))^{-1},
\end{equation}
and
\begin{equation}
\label{R-M-formula}
\begin{split}
R \Bigl(z_1 (1 + M(\mb{z})), \ldots, z_d (1 + M(\mb{z})) \Bigr) & = M(\mb{z}), \\
R \Bigl((1 + M(\mb{z})) z_1, \ldots, (1 + M(\mb{z})) z_d \Bigr) & = M(\mb{z}),
\end{split}
\end{equation}
see \cite{SW97} and Lecture 16 of \cite{Nica-Speicher-book}. Note that the generating function $\ell$ for the classical cumulants can also be defined in a similar way, using the lattice of all partitions, commuting variables, and an exponential moment generating function.
\begin{Lemma}
\label{Lemma:Derivatives}
\br
\begin{enumerate}
\item
For $F, G$ power series in $\mb{z}$,
\[
D_i (F(\mb{z}) G(\mb{z})) = (D_i F)(\mb{z}) G(\mb{z}) + F(0) D_i G(\mb{z}),
\]
where $F(0)$ is the constant term of $F$.
\item
For $z_i = w_i (1 + M(\mb{z}))$,
\[
D_j R(\mb{z}) = (1 + M(\mb{w}))^{-1} D_j M(\mb{w}).
\]
\item
$D_j \eta(\mb{w}) = - D_j (1 + M(\mb{w}))^{-1} = D_j M(\mb{w}) (1 + M(\mb{w}))^{-1}$.
\end{enumerate}
\end{Lemma}

\begin{proof}
Part (a) is straightforward. Part (b) is Lemma~14 of \cite{AnsMulti-Sheffer}. For part (c),
\[
\begin{split}
D_j \eta(\mb{w})
& = - D_j(1 - \eta(\mb{w})) = - D_j (1 + M(\mb{w}))^{-1} \\
& = - D_j \left[ 1 - M(\mb{w}) (1 + M(\mb{w}))^{-1} \right] = D_j M(\mb{w}) (1 + M(\mb{w}))^{-1}. \qedhere
\end{split}
\]
\end{proof}

\br
For $X_1, X_2, \ldots, X_d \in (\mc{A}^{sa}, \Phi)$, we will also use notation
\[
\Cum{X_1, X_2, \ldots, X_d} = \eta_\phi \left[ x_1 x_2 \ldots x_d \right],
\]
where $\phi$ is the joint distribution of $X_1, X_2, \ldots, X_d$ with respect to $\Phi$.

\subsection{Independence}

$\set{x_1, x_2, \ldots, x_d}$ are Boolean (resp., freely) independent with respect to $\phi$ if
\[
\eta\text{ (resp., $R$)} \left[ x_{\vec{u}} \right] = 0
\]
unless all $u(1) = u(2) = \ldots$. This is usually expressed as a ``mixed cumulants are zero'' condition. The condition for free independence in terms of moments is more complicated: whenever $u(1) \neq u(2) \neq \ldots \neq u(n)$ and $P_1, \ldots, P_n$ are polynomials with $\state{P_i(x_{u(i)})} = 0$, then
\[
\state{P_1(x_{u(1)}) P_2(x_{u(2)}) \ldots P_n(x_{u(n)})} = 0.
\]
In contrast, the definition of Boolean independence in terms of moments is more straightforward: again whenever $u(1) \neq u(2) \neq \ldots \neq u(n)$, for any $v(1), v(2), \ldots v(n) \geq 1$
\begin{equation}
\label{Boolean-independence}
\state{x_{u(1)}^{v(1)} x_{u(2)}^{v(2)} \ldots x_{u(n)}^{v(n)}} = \state{x_{u(1)}^{v(1)}} \state{x_{u(2)}^{v(2)}} \ldots \state{x_{u(n)}^{v(n)}}.
\end{equation}
Note that this definition does \emph{not} imply a property like free independence above for general polynomials rather than monomials, since for example
\begin{equation}
\label{Boolean-not-factor}
\state{x_1 1 x_1} = \state{x_1^2} \neq \state{x_1} \state{1} \state{x_1}.
\end{equation}

\begin{Remark}
\label{Remark:Boolean-product}
In Section~\ref{Subsec:Martingales}, we will need the notion of Boolean independence for more general algebras than just the polynomials. In a unital $\ast$-algebra $(\mc{A}, \Phi)$, non-unital $\ast$-subalgebras $\mc{B}_1, \mc{B}_2$ are Boolean independent if for any $X_i \in \mc{B}_{u(i)}$, $u(i) \neq u(i+1)$
\begin{equation}
\label{Boolean-factorization}
\State{X_1 X_2 \ldots X_n} = \State{X_1} \State{X_2} \ldots \State{X_n}.
\end{equation}
If $\Alg{1, \mc{B}_1}$ is the unital subalgebra of $\mc{A}$ generated by $1$ and $\mc{B}_1$, and similarly for $\Alg{1, \mc{B}_2}$, for elements of these subalgebra a factorization of the type \eqref{Boolean-factorization} no longer needs to hold. Nevertheless, the following particular cases will remain true.
\end{Remark}

\begin{Lemma}
\label{Lemma:Boolean-factor}
If $X \in \Alg{1, \mc{B}_1, \mc{B}_2}$, $Y \in \Alg{1, \mc{B}_1}$, $Y' \in \mc{B}_1$, $Z \in \Alg{1, \mc{B}_2}$, then
\begin{gather*}
\State{Y Z} = \State{Y} \State{Z}, \\
\State{X Y' Z} = \State{X Y'} \State{Z}.
\end{gather*}
\end{Lemma}

\subsection{Convolutions and products}

If $\phi, \psi$ are two unital linear functionals on $\mf{C} \langle \mb{x} \rangle$, then $\phi \uplus \psi$ (resp., $\phi \boxplus \phi$), their Boolean (resp., free) convolution, is a unital linear functional on $\mf{C} \langle \mb{x} \rangle$ determined by $\eta_\phi(\mb{z}) + \eta_\psi(\mb{z}) = \eta_{\phi \uplus \psi}(\mb{z})$ (resp., $R_\phi(\mb{z}) + R_\psi(\mb{z}) = R_{\phi \boxplus \psi}(\mb{z})$). See Lecture~12 of \cite{Nica-Speicher-book} for the relation between free convolution and free independence; the relation in the Boolean case is similar.

\br
Unital linear functionals $\phi^{\uplus t}$ (resp., $\phi^{\boxplus t}$) form a Boolean (resp., free) convolution semigroup if $\phi^{\uplus t} \uplus \phi^{\uplus s} = \phi^{\uplus (t+s)}$ (resp., $\phi^{\boxplus t} \boxplus \phi^{\boxplus s} = \phi^{\boxplus (t+s)}$).

\br
If $\set{\phi_1, \phi_2, \ldots, \phi_d}$ are unital linear functionals on $\mf{C}[x]$, their Boolean (resp., free) product is the unital linear functional $\phi = \phi_1 \odot \phi_2 \odot \ldots \odot \phi_d$ (resp., $\phi_1 \ast \phi_2 \ast \ldots \ast \phi_d$) determined by
\[
\eta_\phi(\mb{z}) = \eta_{\phi_1}(z_1) + \ldots + \eta_{\phi_d}(z_d),
\]
resp.
\[
R_\phi(\mb{z}) = R_{\phi_1}(z_1) + \ldots + R{\phi_d}(z_d).
\]

\begin{Remark}
Since we typically consider random variables with zero mean, many of our considerations are valid, or perhaps even more appropriate, for the Fermi convolution of \cite{Oravecz-Fermi}.
\end{Remark}

\section{Boolean Appell polynomials}
\label{Section:Appell}

\subsection{Single variable polynomials}

We first summarize the properties of the single variable Boolean Appell polynomials. The proofs are omitted, as the results follow from the more general multivariate formulas.

\begin{Remark}
If $\mu$ is a positive measure on $\mf{R}$, its Cauchy transform $G_\mu: \mf{C}^+ \rightarrow \mf{C}^-$ is the function
\[
G_\mu(z) = \int_{\mf{R}} \frac{1}{z - x} \,d\mu(x).
\]
Then the function
\[
K_\mu(z) = z - \frac{1}{G_\mu(z)}
\]
has all the properties of the Boolean cumulant generating function. If all the moments of $\mu$ are finite, and so $\mu$ is identified with a linear functional on polynomials, the relation between $K_\mu$ and the Boolean cumulant generating function $\eta$ of $\mu$ is simply
\[
\eta(z) = z K_\mu(1/z).
\]
We will denote the actual Boolean cumulants of $\mu$ by ${\kappa_n}$, so that
\[
\eta(z) = \sum_{n=1}^\infty \kappa_n z^n.
\]
Similarly, $m_n = \mu(x^n)$ are the moments of $\mu$,
\[
M(z) = \sum_{n=1}^\infty m_n z^n.
\]
\end{Remark}

\begin{Prop}
\label{Prop:One-variable-Appell}
Boolean Appell polynomials corresponding to the unital linear functional $\mu$ have the following properties.
\begin{enumerate}
\item
(centering)
$A_0(x) = 1$, $\mu(A_n) = 0$ for $n > 0$.
\item
(differential recursion)
\[
D A_n = A_{n-1},
\]
where
\[
D f(x) = \frac{f(x) - f(0)}{x}.
\]
\item
(recursion)
\[
x A_n = A_{n+1} + \kappa_{n+1},
\]
\item
(explicit formula)
\[
A_n(x) = x^n - \sum_{k=1}^n \kappa_k x^{n-k},
\]
\item
(generating function)
\[
\sum_{n=0}^\infty A_n(x) z^n = \frac{1 - \eta(z)}{1 - x z} = \frac{\frac{1}{z} - K \left( \frac{1}{z} \right)}{\frac{1}{z} - x},
\]
\item
(powers-of-$x$ expansion)
\[
x^n = \sum_{k=0}^n m_k A_{n-k}(x).
\]
\end{enumerate}
\end{Prop}

\begin{Remark}
Boolean Appell polynomials are polynomials of Brenke type \cite{Chihara-Brenke-1}, that is, their generating function has the form $A(z) B(xz)$. Some of the properties in this proposition follow from the results for this more general class. Recall also that, as mentioned in the Introduction, Boolean Appell polynomials are a particular case of the $q$-Appell polynomials considered by Al-Salam. As such, they form a commutative groups under the ``Appell multiplication'' of power series \cite{Appell}, like the classical Appell polynomials and unlike the free Appell polynomials. Another property which holds in the Boolean and classical case but not in the free case is that the lowering operator for the Sheffer class (see the next section) commutes with the lowering operator for the Appell class: $u^{-1}(D) P_n = P_{n-1}$. We will not pursue this approach further.
\end{Remark}

\subsection{Multivariate Boolean Appell polynomials}

Let $\mc{A}$ be a complex $\ast$-algebra with a unital $\ast$-compatible linear functional $\Phi$. For any $n \in \mf{N}$, define a map
\[
A: \left(\mc{A}^{sa}\right)^n \rightarrow \mc{A}, \qquad (X_1, X_2, \ldots, X_n) \mapsto A(X_1, X_2, \ldots, X_n)
\]
by specifying that $A(X_1, X_2, \ldots, X_n)$ is a polynomial in $X_1, X_2, \ldots, X_n$,
\[
D_{X_i} A(X_1, X_2, \ldots, X_n) = \delta_{i 1} A(X_2, \ldots, X_n)
\]
(with $A(\emptyset) = 1$), and
\[
\State{A(X_1, X_2, \ldots, X_n)} = 0.
\]

\br
Since
\[
P(X_1, X_2, \ldots, X_n) = \sum_{i=1}^n X_i D_{X_i} P(X_1, X_2, \ldots, X_n) + \text{ const},
\]
the maps $A$ are determined uniquely (unlike in the free case \cite[Section 3.3]{AnsAppell}, where an extra order condition had to be specified). It is easy to see that each $A$ is multilinear. By taking $\mc{A} = \mf{C} \langle x_1, x_2, \ldots, x_d \rangle$, we get the following definition.

\begin{Defn}
A Boolean Appell polynomial family $\set{A_{\vec{u}}(x_1, x_2, \ldots, x_d)}$ corresponding to a functional $\phi$ on $\mf{C} \langle \mb{x} \rangle$ is a monic polynomial family determined by
\[
D_i A_{(j, \vec{u})}(\mb{x}) = \delta_{ij} A_{\vec{u}}(\mb{x})
\]
and
\[
\state{A_{\vec{u}}(\mb{x})} = \delta_{\vec{u}, \emptyset}.
\]
\end{Defn}

\begin{Prop}
\label{Prop:Appell-expansions}
Let $X_1, X_2, \ldots, X_d \in (\mc{A}^{sa}, \Phi)$, $\phi$ their joint distribution, $\eta(\mb{z})$ its Boolean cumulant generating function, and $\Cum{\cdot}$ its Boolean cumulant functional.
\begin{enumerate}
\item
The generating function for the Boolean Appell polynomials is
\[
H(\mb{x}, \mb{z}) = 1 + \sum_{\vec{u}} A_{\vec{u}}(\mb{x}) z_{\vec{u}} = \left( 1 - \mb{x} \cdot \mb{z} \right)^{-1} (1 - \eta(\mb{z}))= \left( 1 - \mb{x} \cdot \mb{z} \right)^{-1} (1 + M(\mb{z}))^{-1}.
\]
\item
The mutual expansions between monomials and Boolean Appell polynomials are
\begin{equation}
\label{Explicit-form}
A (X_1, X_2, \ldots, X_n) = X_1 X_2 \ldots X_n - \sum_{k=0}^{n-1} X_1 \ldots X_k \Cum{X_{k+1}, \ldots, X_n}.
\end{equation}
and
\[
X_1 X_2 \ldots X_n = A(X_1, X_2, \ldots, X_n) + \sum_{k=0}^{n-1} A(X_1, \ldots, X_k) \state{X_{k+1} \ldots X_n}.
\]
\item
The Boolean Appell polynomials satisfy a recursion relation
\[
X A(X_1, X_2, \ldots, X_n) = A(X, X_1, X_2, \ldots, X_n) + \Cum{X, X_1, X_2, \ldots, X_n}.
\]
\end{enumerate}
\end{Prop}

\begin{proof}
For (a), we note that
\[
D_{x_i} H(\mb{x}, \mb{z})
= D_{x_i} \left( 1 + \sum_{\vec{u}} x_{\vec{u}} z_{\vec{u}} \right) (1 + M(\mb{z}))^{-1}
= z_i H(\mb{x}, \mb{z})
\]
and
\[
\state{H(\mb{x}, \mb{z})}
= \state{\left( 1 + \sum_{\vec{u}} x_{\vec{u}} z_{\vec{u}} \right) (1 + M(\mb{z}))^{-1}}
= (1 + M(\mb{z})) (1 + M(\mb{z}))^{-1} = 1.
\]
Since these conditions determine $A_{\vec{u}}$ uniquely, the result follows.

\br
For (b)
\[
\begin{split}
H(\mb{x}, \mb{z})
& = 1 + \sum_{\vec{u}} A_{\vec{u}}(x) z_{\vec{u}}
= \left(1 + \sum_{\vec{u}} x_{\vec{u}} z_{\vec{u}} \right) (1 - \eta(\mb{z})) \\
& = 1 + \sum_{\vec{u}} x_{\vec{u}} z_{\vec{u}} - \sum_{\vec{u}} \Cum{x_{\vec{u}}} z_{\vec{u}} - \sum_{\vec{u}, \vec{v}} x_{\vec{u}} \Cum{x_{\vec{v}}} z_{\vec{u}} z_{\vec{v}},
\end{split}
\]
which implies
\[
A_{\vec{u}} = x_{\vec{u}} - \sum_{i=0}^{n-1} x_{(u(1), \ldots, u(i))} \Cum{x_{(u(i+1), \ldots, u(n))}}.
\]
Also
\[
\begin{split}
(1 - \mb{x} \cdot \mb{z})^{-1}
& = 1 + \sum_{\vec{u}} x_{\vec{u}} z_{\vec{u}}
= (1 + \sum_{\vec{u}} A_{\vec{u}} z_{\vec{u}}) (1 + M(\mb{z})) \\
& = 1 + \sum_{\vec{u}} A_{\vec{u}} z_{\vec{u}} + \sum_{\vec{u}} \state{x_{\vec{u}}} z_{\vec{u}} + \sum_{\vec{u}, \vec{v}} A_{\vec{u}} \state{x_{\vec{v}}} z_{\vec{u}} z_{\vec{v}},
\end{split}
\]
which implies
\[
x_{\vec{u}} = A_{\vec{u}} + \sum_{i=0}^{n-1} A_{(u(1), \ldots, u(i))} \state{x_{(u(i+1), \ldots, u(n))}}.
\]

\br
Finally, part (c) follows from the first expansion in part (b).
\end{proof}

\br
The proof of Proposition~\ref{Prop:One-variable-Appell} now follows from the observation that
\[
A_n(x_i) = A_{\underbrace{1, 1, \ldots, 1}_{n \text{ times}}}(\mb{x}).
\]

\subsection{Boolean martingales}
\label{Subsec:Martingales}

\begin{Prop}(Boolean binomial properties).
\label{Prop:Boolean-binomial}
\begin{enumerate}
\item
If the variables $\set{x_i}$ are Boolean independent and
\[
k = \min (i | x_{u(i+1)} = x_{u(i+2)} = \ldots = x_{u(n)} = x),
\]
then
\[
A_{\vec{u}}(\mb{x})
= x_{(u(1), \ldots, u(k))} A_{n-k}(x).
\]
More generally, if $X_k$ is Boolean independent from $X_{k+1}, \ldots, X_n$, then
\[
A(X_1, X_2, \ldots, X_n) = X_1 X_2 \ldots X_k A(X_{k+1}, \ldots, X_n).
\]
\item
If $\set{X_i}, \set{Y_i} \in \mc{A}$ are Boolean independent with respect to $\Phi$, then
\[
\begin{split}
A(X_1 + Y_1, X_2 + Y_2, \ldots, X_n + Y_n)
& = A(X_1, X_2, \ldots, X_n) + A(Y_1, Y_2, \ldots, Y_n) \\
&\quad + \sum_{k=1}^{n-1} (X_1 + Y_1) \ldots (X_{k-1} + Y_{k-1}) Y_k A(X_{k+1}, \ldots, X_n) \\
&\quad + \sum_{k=1}^{n-1} (X_1 + Y_1) \ldots (X_{k-1} + Y_{k-1}) X_k A(Y_{k+1}, \ldots, Y_n).
\end{split}
\]
\end{enumerate}
\end{Prop}

\begin{proof}
Using the explicit form formula~\eqref{Explicit-form},
\[
\begin{split}
A(X_1, X_2, \ldots, X_n)
& = X_1 X_2 \ldots X_n - \sum_{i=0}^{n-1} X_1 \ldots X_i \Cum{X_{i+1}, \ldots, X_n} \\
& = X_1 X_2 \ldots X_n - \sum_{i=k}^{n-1} X_1 \ldots X_i \Cum{X_{i+1}, \ldots, X_n} \\
& = X_1 X_2 \ldots X_n - X_1 \ldots X_k \sum_{i=k}^{n-1} X_{k+1} \ldots X_i \Cum{X_{i+1}, \ldots, X_n} \\
& = X_1 X_2 \ldots X_k A(X_{k+1}, \ldots, X_n).
\end{split}
\]
Therefore, using the multi-linearity of the Boolean Appell polynomials
\[
\begin{split}
A(X_1 + Y_1, X_2 + Y_2, \ldots, X_n + Y_n)
& = A(X_1, X_2, \ldots, X_n) + A(Y_1, Y_2, \ldots, Y_n) \\
&\quad + \sum_{k=1}^{n-1} A(X_1 + Y_1, \ldots, X_{k-1} + Y_{k-1}, Y_k, X_{k+1}, \ldots, X_n) \\
&\quad + \sum_{k=1}^{n-1} A(X_1 + Y_1, \ldots, X_{k-1} + Y_{k-1}, X_k, Y_{k+1}, \ldots, Y_n) \\
& = A(X_1, X_2, \ldots, X_n) + A(Y_1, Y_2, \ldots, Y_n) \\
&\quad + \sum_{k=1}^{n-1} (X_1 + Y_1) \ldots (X_{k-1} + Y_{k-1}) Y_k A(X_{k+1}, \ldots, X_n) \\
&\quad + \sum_{k=1}^{n-1} (X_1 + Y_1) \ldots (X_{k-1} + Y_{k-1}) X_k A(Y_{k+1}, \ldots, Y_n).
\end{split}
\]
\end{proof}

\begin{Remark}
Let $\mc{A}$ be a complex $\ast$-algebra, $\Phi$ a state on it, and $\mc{B} \subset \mc{A}$ a subalgebra. We can form Hilbert spaces $L^2(\mc{B}, \Phi) \subset L^2(\mc{A}, \Phi)$, and let $\State{\cdot | \mc{B}}$ denote the projection on the subspace $L^2(\mc{B}, \Phi)$. If in fact $\State{\mc{A} | \mc{B}} \subset \mc{A}$, this projection is called a (state-preserving) conditional expectation. This is always the case if $\mc{A}, \mc{B}$ are von Neumann algebras and $\Phi$ is a trace, however states in Boolean theory typically are not tracial. Note that if $\Phi$ is faithful on $\mc{B}$, $\State{Y | \mc{B}}$ is characterized by the property that
\begin{equation}
\label{Projection}
\State{X \state{Y | \mc{B}}} = \State{X Y}
\end{equation}
for any $X \in \mc{B}$.
\end{Remark}

\begin{Prop}
If $\Phi$ is a faithful state, $\set{X_i, Y_i | i = 1, \ldots, n} \subset (\mc{A}^{sa}, \Phi)$, $\set{X_i} \subset \mc{B}$ and $\set{Y_i}$ are Boolean independent of the subalgebra $\mc{B}$ in the sense of Remark~\ref{Remark:Boolean-product}, then
\[
\State{A(X_1 + Y_1, X_2 + Y_2, \ldots, X_n + Y_n) | \mc{B}} = A(X_1, X_2, \ldots, X_n).
\]
In particular, Boolean Appell polynomials are Boolean martingales
\[
\State{A_n(X + Y) | X} = A_n(X)
\]
and
\[
\State{\left. \frac{1 - (\eta_X(z) + \eta_Y(z))}{1 - (X + Y)z} \right| \mc{B}} = \frac{1 - \eta_X(z)}{1 - X z}.
\]
\end{Prop}

\begin{proof}
Using Proposition~\ref{Prop:Boolean-binomial}, Lemma~\ref{Lemma:Boolean-factor}, and the fact that the Boolean Appell polynomials are centered, for any $X \in \mc{B}$
\[
\begin{split}
\State{X A(X_1 + Y_1, \ldots, X_n + Y_n)}
& = \State{X A(X_1, X_2, \ldots, X_n)} + \State{X A(Y_1, Y_2, \ldots, Y_n)} \\
&\quad + \sum_{k=1}^{n-1} \State{X (X_1 + Y_1) \ldots (X_{k-1} + Y_{k-1}) Y_k A(X_{k+1}, \ldots, X_n)} \\
&\quad + \sum_{k=1}^{n-1} \State{X (X_1 + Y_1) \ldots (X_{k-1} + Y_{k-1}) X_k A(Y_{k+1}, \ldots, Y_n)} \\
& = \State{X A(X_1, X_2, \ldots, X_n)}.
\end{split}
\]
\end{proof}

\br
The preceding proposition is closely related to the Markov property for the processes with Boolean independent increments (see Section 4.2 of \cite{Franz-Unification}), which will be investigated in a future paper.

\begin{Remark}[Fock space representation and the Kailath-Segall polynomials]
While the Appell polynomials are sometimes called Wick products, it is more appropriate to reserve that name for the following objects. Let $\mc{A}_0$ be a $\ast$-algebra with state $\psi$. The operators $W(f_1, f_2, \ldots, f_k)$ in the algebra of symbols generated by $\set{X(f)}$, for $f_i \in \mc{A}_0^{sa}$, are defined via the relations
\begin{align}
\label{Wick-recursion}
W(f) & = X(f), \notag \\
W(f, f_1) & = X(f) W(f_1) - W(f f_1) - \psi \left[ f f_1 \right] I, \\
W(f, f_1, \ldots, f_n) & = X(f) W(f_1, f_2, \ldots, f_n) - W(f f_1, f_2, \ldots, f_n) \notag
\end{align}
for $n \geq 2$. Compare with Section 3.7 of \cite{AnsAppell}, and note the use of the Boolean annihilation operator, acting only on the first level of the Fock space (see below), as in Section~\ref{Subsec:Process} and Proposition~\ref{Prop:General-Boolean-Fock}. $W(f_1, f_2, \ldots, f_n)$ is a polynomial in the variables
\[
\set{X\bigl(\prod_{i \in S} f_i \bigr) | S \subset \set{1, 2, \ldots, n}},
\]
called the (Boolean) Kailath-Segall polynomial. However,
\[
A(X(f_1), X(f_2), \ldots, X(f_n)) = \sum_{\substack{\pi \in \Int(n) \\ \pi = (B_1, B_2, \ldots, B_k)}} W \Bigl(\prod_{i(1) \in B_1} f_{i(1)}, \prod_{i(2) \in B_2} f_{i(2)}, \ldots, \prod_{i(k) \in B_k} f_{i(k)} \Bigr)
\]
and in particular is a polynomial in $\set{X(f_1), X(f_2), \ldots, X(f_n)}$ only. Note that this expansion is exactly the same as in the free case \cite[Proposition 3.22]{AnsAppell}, but of course in terms of different polynomials. The expansions in Proposition~\ref{Prop:Appell-expansions} can now be derived from the more basic expansions in terms of the Kailath-Segall polynomials.

\br
Denote $\mc{H} = L^2(\mc{A}_0, \psi)$. Let
\[
\Falg(\mc{H}) = \mf{C} \Omega \oplus \mc{H} \oplus \mc{H}^{\otimes 2} \oplus \ldots
\]
be its algebraic full Fock space. If we represent all these operators on $\Falg(\mc{H})$ via
\begin{equation}
\label{Wick-representation}
W(f_1, f_2, \ldots, f_n) \Omega = f_1 \otimes f_2 \otimes \ldots \otimes f_n,
\end{equation}
and put an inner product on $\Falg(\mc{H})$ in such a way that different components of $\Falg(\mc{H})$ are orthogonal and each $X(f)$ is symmetric, it follows from the relations~\eqref{Wick-representation} and \eqref{Wick-recursion} that $\norm{\zeta} = 0$ for $\zeta \in \mc{H}^{\otimes k}$, $k \geq 2$. So in this case
\[
A(X(f_1), X(f_2), \ldots, X(f_n)) \Omega = f_1 f_2 \ldots f_n.
\]
\end{Remark}

\section{Boolean Meixner polynomials}
\label{Section:Meinxer}

\subsection{Boolean Meixner states}
Since we have found Boolean Appell polynomials to have generating functions of the form $\left( 1 - \mb{x} \cdot \mb{z} \right)^{-1} (1 - \eta(\mb{z}))$, we now define Boolean Sheffer polynomials to be polynomial families with generating functions of the form
\[
\left( 1 - \mb{x} \cdot \mb{V}(\mb{z}) \right)^{-1} (1 - \eta(\mb{V}(\mb{z})))
\]
for some $d$-tuple of non-commutative power series $\mb{V}$. Recall \cite{AnsAppell} that free Sheffer polynomial families have generating functions of the form
\[
\Bigl( 1 - \mb{x} \cdot \mb{U}(\mb{z}) + R(\mb{U}(\mb{z})) \Bigr)^{-1}.
\]
We now show that these are the same.

\begin{Prop}
\label{Prop:Free-Boolean-change}
Let $\phi$ be a functional, and $R, \eta$ its free, respectively, Boolean, cumulant generating functions. Then
\[
\Bigl( 1 - \mb{x} \cdot \mb{U} + R(\mb{U}) \Bigr)^{-1} = \left( 1 - \mb{x} \cdot \mb{V} \right)^{-1} (1 - \eta(\mb{V})),
\]
where
\begin{equation}
\label{U-V}
\mb{U} = (1 + M(\mb{V})) \mb{V}
\end{equation}
and
\begin{equation}
\label{V-U}
\mb{V} = (1 + R(\mb{U}))^{-1} \mb{U}
\end{equation}
are two $d$-tuples of power series.
\end{Prop}

\begin{proof}
We first note that an application of the identity~\eqref{R-M-formula} shows that equation~\eqref{U-V} is equivalent to
\[
R(\mb{U}) = M(\mb{V}),
\]
which is equivalent to \eqref{V-U} since
\[
V_i = (1 + M(\mb{V}))^{-1} (1 + M(\mb{V})) V_i = (1 + R(\mb{U}))^{-1} (1 + M(\mb{V})) V_i.
\]
Therefore
\[
\begin{split}
\Bigl( 1 - \mb{x} \cdot \mb{U} + R(\mb{U}) \Bigr)^{-1}
& = \Bigl( (1 + R(\mb{U})) \left( 1 - (1 + R(\mb{U}))^{-1} (\mb{x} \cdot \mb{U}) \right) \Bigr)^{-1} \\
& = \Bigl( 1 - (1 + R(\mb{U}))^{-1} (\mb{x} \cdot \mb{U}) \Bigr)^{-1} (1 + R(\mb{U}))^{-1} \\
& = \Bigl( 1 - (1 + M(\mb{V}))^{-1} (\mb{x} \cdot (1 + M(\mb{V})) \mb{V}) \Bigr)^{-1} (1 + M(\mb{V}))^{-1} \\
& = \left( 1 - \mb{x} \cdot \mb{V} \right)^{-1} (1 - \eta(\mb{V})). \qedhere
\end{split}
\]
\end{proof}

\br
It follows that the free and Boolean Meixner polynomial families coincide, and therefore so do the classes of free and Boolean Meixner states. Recall the characterizations of the free Meixner states. Theorem~3 of \cite{AnsFree-Meixner} described such states on $\mf{R} \langle \mb{x} \rangle$. The following proposition is the corresponding result for states on $\mf{C} \langle \mb{x} \rangle$. Its proof is a straightforward translation.

\begin{Prop}
\label{Prop:Free-Meixner}
Let $\phi$ be a state on $\mf{C} \langle \mb{x} \rangle$ with a monic orthogonal polynomial system (MOPS), zero means and identity covariance. $\phi$ is a free Meixner state if and only if any one of the following equivalent conditions holds.
\begin{enumerate}
\item
The polynomials with the generating function
\begin{equation*}
\sum_{\abs{\vec{u}} \geq 0} P_{\vec{u}}(\mb{x}) z_{\vec{u}} = \Bigl( 1 - \mb{x} \cdot \mb{U}(\mb{z}) + R(\mb{U}(\mb{z})) \Bigr)^{-1}
\end{equation*}
for some $\mb{U}$ are a MOPS for $\phi$, where $R$ is the free cumulant generating function of $\phi$. In this case necessarily
\[
(D_i R)(\mb{U}(\mb{z})) = z_i,
\]
so that $\mb{U} = (\mb{D} R)^{\langle -1 \rangle}$.
\item
There exist Hermitian $d \times d$ matrices $T_i$ and a diagonal $d^2 \times d^2$ matrix $C$ with $I + C \geq 0$ and $(T_i \otimes I) C = C (T_i \otimes I)$ such that $\phi$ has a representation $\phi_{C, \set{T_i}}$ as a Fock state in the sense of Proposition~\ref{Prop:Monic-states}. Here $\mc{C}^{(1)} = I$, $\mc{C}^{(k)} = C \otimes I^{\otimes (k-2)}$ for $k \geq 2$, $\mc{T}_i^{(0)} = 0$, $\mc{T}_i^{(k)} = T_i \otimes I^{\otimes (k-1)}$ for $k \geq 1$.
\item
Denoting the entries of the matrices $T_k, C$ by $B^k_{ij}, C_{ij}$, respectively, the free cumulant generating function of $\phi$ satisfies, for each $i, j$, a (non-commutative) second-order partial differential equation
\begin{equation}
\label{PDE}
D_i D_j R(\mb{z}) = \delta_{ij} + \sum_{k=1}^d B_{ij}^k D_k R(\mb{z}) + C_{ij} D_i R(\mb{z}) D_j R(\mb{z}).
\end{equation}
\item
There is a family of polynomials $\set{P_{\vec{u}}}$ such that $\state{P_{\vec{u}}} = 0$ for all $\vec{u} \neq \emptyset$ and they satisfy a recursion relation
\begin{align*}
x_i & = P_i, \\
x_i P_{j} &= P_{(i,j)} + \sum_{k=1}^d B_{ij}^{k} P_{k} + \delta_{ij}, \\
x_i P_{(j, \vec{u})} &= P_{(i, j, \vec{u})} + \sum_{k=1}^d B_{ij}^{k} P_{(k, \vec{u})} + \delta_{ij} (1 + C_{i, u(1)}) P_{\vec{u}}.
\end{align*}
\end{enumerate}
\end{Prop}

\br
We now show that, while the Boolean and free Meixner states coincide, and so the free characterizations above hold for the Boolean Meixner states, the Boolean versions of these properties hold as well.

\begin{Prop}
\label{Prop:Boolean-Meixner}
The following are equivalent to the conditions in Propositions~\ref{Prop:Free-Meixner}.
\begin{enumerate}
\item
The polynomials with the generating function
\[
\left( 1 - \mb{x} \cdot \mb{V} \right)^{-1} (1 - \eta(\mb{V}))
\]
are a MOPS for $\phi$, where
\[
(D_i K)(\mb{V}(\mb{z})) = z_i.
\]
so that $\mb{V} = (\mb{D} K)^{\langle -1 \rangle}$.
\item
The Boolean cumulant generating function of $\phi$ satisfies, for each $i, j$,
\[
D_i D_j \eta(\mb{w}) = \delta_{ij} + \sum_k B_{ij}^k D_k \eta(\mb{w}) + (1 + C_{ij}) D_i \eta(\mb{w}) D_j \eta(\mb{w}).
\]
\end{enumerate}
\end{Prop}

\begin{proof}
For part (a), we use the known relation in the free case and the change of variables in Proposition~\ref{Prop:Free-Boolean-change}:
\[
(D_i R)(\mb{U}) = z_i
\]
is equivalent to
\[
R(\mb{U}) = \sum_i U_i z_i,
\]
so that
\[
M(\mb{V}) = \sum_i (1 + M(\mb{V})) V_i z_i,
\]
and
\[
\eta(\mb{V}) = 1 - (1 + M(\mb{V}))^{-1} = (1 + M(\mb{V}))^{-1} M(\mb{V}) = \sum_i V_i z_i,
\]
which finally is equivalent to
\[
(D_i \eta)(\mb{V}) = z_i.
\]
For part (b), we start with the characterization
\[
D_i D_j R(\mb{z}) = \delta_{ij} + \sum_k B_{ij}^k D_k R(\mb{z}) + C_{ij} D_i R(\mb{z}) D_j R(\mb{z}).
\]
Multiplying by $z_i$ and summing over $i$,
\[
D_j R(\mb{z}) = z_j + \sum_{k, i} B_{ij}^k z_i D_k R(\mb{z}) + \sum_{i} C_{ij} z_i D_i R(\mb{z}) D_j R(\mb{z}).
\]
Substituting $z_i = w_i (1 + M(\mb{w}))$ and using Lemma~\ref{Lemma:Derivatives},
\[
\begin{split}
(1 + M(\mb{w}))^{-1} D_j M(\mb{w})
& = w_j (1 + M(\mb{w})) + \sum_{k, i} B_{ij}^k w_i (1 + M) (1 + M)^{-1} D_k M(\mb{w}) \\
&\quad + \sum_{i} C_{ij} w_i (1 + M) (1 + M)^{-1} D_i M(\mb{w}) (1 + M(\mb{w}))^{-1} D_j M(\mb{w}) \\
& = w_j (1 + M(\mb{w})) + \sum_{k, i} B_{ij}^k w_i D_k M(\mb{w}) \\
&\quad + \sum_{i} C_{ij} w_i D_i M(\mb{w}) (1 + M(\mb{w}))^{-1} D_j M(\mb{w}).
\end{split}
\]
Multiplying by $(1 + M(\mb{w}))^{-1}$ on the right, we get
\[
\begin{split}
& (1 + M(\mb{w}))^{-1} D_j M(\mb{w}) (1 + M(\mb{w}))^{-1} \\
&\qquad = w_j + \sum_{k, i} B_{ij}^k w_i D_k M(\mb{w}) (1 + M(\mb{w}))^{-1} \\
&\qquad\quad + \sum_{i} C_{ij} w_i D_i M(\mb{w}) (1 + M(\mb{w}))^{-1} D_j M(\mb{w}) (1 + M(\mb{w}))^{-1}
\end{split}
\]
so that, using Lemma~\ref{Lemma:Derivatives} again,
\[
(1 + M(\mb{w}))^{-1} D_j \eta(\mb{w}) = w_j + \sum_{k, i} B_{ij}^k w_i D_k \eta(\mb{w}) + \sum_{i} C_{ij} w_i D_i \eta(\mb{w}) D_j \eta(\mb{w}).
\]
Applying $D_i$, we finally get
\[
D_i(1 + M(\mb{w}))^{-1} D_j \eta(\mb{w}) + D_i D_j \eta(\mb{w}) = \delta_{ij} + \sum_k B_{ij}^k D_k \eta(\mb{w}) + C_{ij} D_i \eta(\mb{w}) D_j \eta(\mb{w}),
\]
or
\[
D_i D_j \eta(\mb{w}) = \delta_{ij} + \sum_k B_{ij}^k D_k \eta(\mb{w}) + (1 + C_{ij}) D_i \eta(\mb{w}) D_j \eta(\mb{w}). \qedhere
\]
\end{proof}

\subsection{Belinschi-Nica transformation}

Belinschi and Nica \cite{Belinschi-Nica-B_t,Belinschi-Nica-Free-BM} have considered a ``remarkable transformation''
\[
\mf{B}_t(\mu) = \left(\mu^{\boxplus(1+t)}\right)^{\uplus(1/(1+t))},
\]
which has a number of surprising properties, for example the relation to the multiplicative free convolution. A number of examples computed in the first of these papers are explained by the following proposition.

\begin{Prop}
\label{Prop:Free/Boolean}
\br
\begin{enumerate}
\item
The free / Boolean Meixner class is closed under free and Boolean convolution powers, and consequently under the operation $\mf{B}_t$. In fact,
\[
\mf{B}_t(\phi_{C, \set{T_i}}) = \phi_{tI + C, \set{T_i}}
\]
\item
Every one-dimensional free / Boolean Meixner distribution can be obtained from a Bernoulli distribution by the application of an appropriate $\mf{B}_t$.
\end{enumerate}
\end{Prop}

\begin{proof}
Let $\phi = \phi_{C, \set{T_i}}$ be a free Meixner state. As shown in Section~3.1 of \cite{AnsFree-Meixner}, the MOPS for $\phi^{\boxplus t}$ satisfy the recursion relations
\begin{align*}
x_i & = P_i, \\
x_i P_j & = P_{(i, j)} + \sum_{k=1}^d B_{ij}^k P_{k} + \delta_{ij} t, \\
x_i P_{(j, \vec{u})} & = P_{(i, j, \vec{u})} + \sum_{k=1}^d B_{ij}^k P_{(k, \vec{u})} + \delta_{ij}(t + C_{i, u(1)}) P_{\vec{u}}.
\end{align*}
which shows that $\phi^{\boxplus t}$ is a dilated free Meixner state (with covariance $t I$). On the other hand, it follows from Corollary~\ref{Cor:Boolean-semigroup} that the MOPS for $\phi^{\uplus t}$ satisfy the recursions relations
\begin{align*}
x_i & = P_i, \\
x_i P_j & = P_{(i, j)} + \sum_{k=1}^d B_{ij}^k P_{k} + \delta_{ij} t, \\
x_i P_{(j, \vec{u})} & = P_{(i, j, \vec{u})} + \sum_{k=1}^d B_{ij}^k P_{(k, \vec{u})} + \delta_{ij}(1 + C_{i, u(1)}) P_{\vec{u}},
\end{align*}
so $\phi^{\uplus t}$ is also a dilated free Meixner state. Finally, it follows that $\mf{B}_t(\phi) = \left(\phi^{\boxplus(1+t)}\right)^{\uplus(1/(1+t))}$ also has a MOPS, and they satisfy recursion relations
\begin{align*}
x_i & = P_i, \\
x_i P_j & = P_{(i, j)} + \sum_{k=1}^d B_{ij}^k P_{k} + \delta_{ij}, \\
x_i P_{(j, \vec{u})} & = P_{(i, j, \vec{u})} + \sum_{k=1}^d B_{ij}^k P_{(k, \vec{u})} + \delta_{ij}(1 + t + C_{i, u(1)}) P_{\vec{u}},
\end{align*}
so that $\mf{B}_t(\phi) = \phi_{tI + C, \set{T_i}}$, which completes the proof of (a). For (b), we note that in the one-dimensional case,
\[
\mf{B}_{(1+c)} \phi_{-1, b} = \phi_{c, b}
\]
and the free Meixner distributions with $c = -1$ are exactly the Bernoulli distributions, specifically
\begin{equation}
\label{Binomial}
\frac{1}{1 + \beta^2} \delta_{\beta} + \frac{\beta^2}{1 + \beta^2} \delta_{-1/\beta}
\end{equation}
with $b = \beta - 1/\beta$.
\end{proof}

\begin{Remark}
The same argument shows that every ``simple quadratic'' free / Boolean Meixner state (a class of states with $C_{ij} \equiv c$, investigated at the end of \cite{AnsFree-Meixner}) can be obtained by the application of an appropriate $\mf{B}_t$ from a free Meixner state with $C = -I$. Moreover, it is easy to deduce from Proposition~9 of \cite{AnsFree-Meixner} that a tracial state of the form $\phi_{-I, \set{T_i}}$ factors through to a multinomial distribution on polynomials in commuting variables. Unfortunately, Boolean convolution, and so the multivariate version of $\mf{B}_t$, does not preserve the trace property.

\br
On the other hand, any tracial ``simple quadratic'' free Meixner state with $c < 0$ (and, in the limit, with $c=0$) can be obtained from a multinomial distribution by a free convolution power with $t = (-1/c)$ followed by a dilation by $\frac{1}{\sqrt{t}}$.
\end{Remark}

\begin{Remark}[Bercovici-Pata bijection]
\label{Remark:BP-bijection}
The classes of distributions infinitely divisible in the classical, free, and Boolean sense are all isomorphic, and (in the one-dimensional case) the corresponding measures have the same domains of attraction \cite{BerPatDomains}. In the multivariate combinatorial case we are considering, the Boolean-to-free correspondence is simply
\[
\phi \mapsto \psi, \qquad \text{ where } \eta_\phi(\mb{z}) = R_\psi(\mb{z})
\]
(of course, the key issue is the infinite divisibility of $\phi, \psi$, in the state rather than a linear functional sense). In general, there is no simple formula for this correspondence (but see \cite{Belinschi-Nica-Eta}). However, it follows from part (b) of Proposition~\ref{Prop:Boolean-Meixner} that this bijection takes the free / Boolean Meixner class to itself, and maps
\[
\phi_{C, \set{T_i}} \mapsto \phi_{I + C, \set{T_i}}.
\]
Note that all free Meixner states are infinitely divisible in the Boolean sense, and the ones infinitely divisible in the free sense are exactly those for $I + C \geq 0$.
\end{Remark}

\begin{Ex}[Free and Boolean product states]
\label{Example:Product-states}
Let $\phi_i$, $i=1, 2, \ldots, d$ be one-dimensional states (measures), whose monic orthogonal polynomials satisfy the recursions
\[
x_i P^{(i)}_n(x_i) = P^{(i)}_{n+1}(x_i) + \beta^{(i)}_n P^{(i)}_n(x_i) + \gamma^{(i)}_n P^{(i)}_{n-1}(x_i)
\]
Their free product state
\[
\phi = \phi_1 \ast \phi_2 \ast \ldots \ast \phi_d,
\]
has a MOPS, which satisfy recursion relations with the following coefficients: for $u(1) \neq i$ and $k\geq 0$,
\[
B^i_{(i^k, \vec{u}), (i^k, \vec{u})} = \beta^{(i)}_k
\]
(where in this example only, $i^k$ denotes $(i, i, \ldots, i)$ repeated $k$ times) and the rest zero, and for $k \geq 1$,
\[
\mc{C}_{(i^k, \vec{u})} = \gamma^{(i)}_k.
\]
In particular, if each $\phi_j$ is a free Meixner state, so that $\beta_0 = 0$, $\gamma_1 = 1$,
\[
\beta_1 = \beta_2 = \ldots = \beta
\]
and
\[
\gamma_2 = \gamma_3 = \ldots = \gamma,
\]
then
\[
B^i_{(j, \vec{u}), (k, \vec{v})} = \delta_{ijk} \delta_{\vec{u}, \vec{v}} \beta^{(i)}
\]
and
\[
\mc{C}_{(i, j, \vec{u})} =
\begin{cases}
\gamma^{(i)}, & i = j, \\
1, & i \neq j,
\end{cases}
\]
so that $C_{ij} = \delta_{ij} (\gamma^{(i)} - 1)$ and $\phi$ is also a free Meixner state.

\br
The Boolean product
\[
\psi = \phi_1 \odot \phi_2 \odot \ldots \odot \phi_d
\]
also has a MOPS, which satisfy recursion relations with coefficients
\[
B^i_{i^k, i^k} = \beta^{(i)}_k,
\]
\[
\mc{C}_{i^k} = \gamma^{(i)}_k
\]
and the rest zero. More precisely (see the proof of Theorem~2 of \cite{AnsMonic}) $\mc{C}_{(i, j^k)} = 0$ for $i \neq j$, and other $\mc{C}_{(\vec{u}, j^k)}$ and $B^i_{(\vec{u}, j^k), \vec{v}}$ for which some $u(i) \neq j$ can be defined arbitrarily. For the product of free Meixner states, we can therefore set
\[
B^i_{(j, \vec{u}), (k, \vec{v})} = \delta_{ijk} \delta_{\vec{u}, \vec{v}} \beta^{(i)}
\]
and
\[
\mc{C}_{(i, j, \vec{u})} = \delta_{ij} \gamma^{(i)},
\]
so that $C_{ij} = \delta_{ij} \gamma^{(i)} - 1$, and $\psi$ is a free Meixner state. Note that this is consistent with the results in Remark~\ref{Remark:BP-bijection}, since that bijection takes Boolean products to free products.
\end{Ex}

\subsection{Conditional freeness}

Another place where free Meixner distributions appear is the theory of conditional freeness \cite{BLS96}. The objects in this theory are algebras with pairs of functionals on them. Conditionally free product induces a convolution on pairs of measures,
\[
(\mu, \nu) = (\mu_1, \nu_1) \boxplus (\mu_2, \nu_2),
\]
so that $\nu = \nu_1 \boxplus \nu_2$ and $\mu$ is determined via
\[
C_{(\mu, \nu)}(z) = C_{(\mu_1, \nu_1)}(z) + C_{(\mu_2, \nu_2)}(z),
\]
where $C_{(\mu, \nu)}$ is a formal power series determined by
\[
C \Bigl[z(1 + M_\nu(z)) \Bigr] (1 + M_\mu(z)) = M_\mu(z) (1 + M_\nu(z)).
\]
Note that if $\nu_1 = \nu_2 = \delta_0$, then $C(z) = \eta_{\mu}(z)$ so $\mu = \mu_1 \uplus \mu_2$, and if $\mu_1 = \nu_1$, $\mu_2 = \nu_2$ then $C(z) = R_{\mu}(z)$ so that $\mu = \mu_1 \boxplus \mu_2$.

\br
In the central and Poisson limit theorems for the conditionally free convolution, the (first components of the) limiting distributions are free Meixner distributions (see also \cite{Bryc-Wesolowski-Bi-Poisson} for a related result). Indeed, these theorems involve measures for which $C_{(\mu, \nu)}(z)$ is equal to $R_\nu(z)$ ($=z^2$ in the central and $=\frac{z^2}{1-z}$ in the Poisson cases), or more generally is a constant multiple of it. In the first case
\[
R_\nu \Bigl[z(1 + M_\nu(z)) \Bigr] (1 + M_\mu(z)) = M_\nu(z) (1 + M_\mu(z)) = M_\mu(z) (1 + M_\nu(z)),
\]
so that $\mu = \nu = $ the semicircle law in the central and the Marchenko-Pastur law in the Poisson limit theorem. Note that both of these laws are free Meixner. In the more general case when
\[
C_{(\mu, \nu)}(z) = \alpha R_\nu(z),
\]
we get
\[
\alpha M_\nu(z) (1 + M_\mu(z)) = M_\mu(z) (1 + M_\nu(z)),
\]
whence
\[
\eta_\mu(z) = \alpha \eta_{\nu}(z),
\]
so that
\[
\mu = \nu^{\uplus \alpha}.
\]
It remains to note that Boolean convolution powers of the semicircle or the Marchenko-Pastur law are free Meixner. Indeed, it follows that if $C_{(\mu, \nu)}(z) = \alpha R_\nu(z)$ and $\nu$ is free Meixner, then so is $\mu$.

\br
Similar calculations explain the appearance of the free Meixner laws as limit laws in \cite{Boz-Wys} and \cite{Krystek-Yoshida-t}. Further properties of the Appell and Sheffer-type objects in the theory of conditional freeness will be explored in a future paper.

\subsection{Laha-Lukacs property}

Laha and Lukacs \cite{Laha-Lukacs} proved that the classical Meixner distributions are characterized by a quadratic regression property. Bo{\.z}ejko and Bryc \cite{Boz-Bryc} proved that the identical property, in the context of free probability, characterizes the free Meixner distributions. We now show that in the Boolean theory, this characterization fails. Instead, the Boolean Laha-Lukacs property characterizes only the Bernoulli distributions. Notice that these can also be interpreted as the Boolean versions of the Poisson distributions, with the symmetric Bernoulli distribution being the analog of the normal. Note also that, while conditional expectations in general may not exist, the expressions below are well-defined. Denote
\[
\mathrm{Var} \left[X | \mc{B} \right] = \State{\bigl(X - \State{X | \mc{B}}\bigr)^2 | \mc{B}}.
\]
In the non-tracial case, this need not equal $\State{X^2 | \mc{B}} - \State{X | \mc{B}}^2.$

\begin{Prop}
Suppose $X, Y$ are Boolean independent (with respect to $\Phi$), self-adjoint, non-degenerate and there are numbers $\alpha, \alpha_0, C, a, b \in \mf{R}$ such that
\[
\State{X | X + Y} = \alpha(X + Y) + \alpha_0
\]
and
\begin{equation}
\label{Quadratic-regression}
\mathrm{Var}\left[X | X + Y \right] = C \Bigl(1 + a(X + Y) + b(X + Y)^2 \Bigr).
\end{equation}
Then $X, Y$ have Bernoulli distributions.
\end{Prop}

\begin{proof}
The proof of Theorem 3.2 in \cite{Boz-Bryc} goes through verbatim until Lemma 4.1. Briefly, we may assume that $\State{X} = \State{Y} = 0$, and $\State{X^2} + \State{Y^2} = 1$, so that $\alpha_0 = 0$ and $\alpha = \State{X^2}$. Denote $\beta = \State{Y^2}$, $\mf{S} = X + Y$ and $\mf{V} = \beta X - \alpha Y$. Then
\begin{equation}
\label{X-Y-S}
\kappa_n(X) = \alpha \kappa_n(\mf{S}) \qquad \text{and} \qquad \kappa_n(Y) = \beta \kappa_n(\mf{S}),
\end{equation}
and
\begin{equation}
\label{V-V-S}
\Cum{\underbrace{\mf{S}, \ldots, \mf{S}}_{n-2 \text{ times}}, \mf{V}, \mf{V}} = \alpha \beta \kappa_n(\mf{S})
\end{equation}
(note an order change from \cite{Boz-Bryc}) for $n \geq 2$. Also,
\[
\State{\mf{V}^2 | \mf{S}} = \State{(X - \alpha \mf{S})^2 | S} = \State{(X - \State{X | \mf{S}})^2 | \mf{S}} = \mathrm{Var} \left[X | \mf{S} \right].
\]
Finally, after putting in normalizations and using equation~\eqref{Projection}, it follows from equation~\eqref{Quadratic-regression} that
\[
\State{\mf{S}^{n} \mf{V}^2} = \frac{\alpha \beta}{1 + b} \Bigl(\State{\mf{S}^n} + a \State{\mf{S}^{n+1}} + b \State{\mf{S}^{n+2}} \Bigr)
\]

\br
Now using equations~\eqref{Cumulants-definition} and \eqref{V-V-S},
\[
\begin{split}
\State{\mf{S}^{n} \mf{V}^2}
& = \State{\mf{S}^{n} \mf{V}} \Cum{\mf{V}} + \State{\mf{S}^n} \Cum{\mf{V}, \mf{V}} + \sum_{\substack{\pi \in \Int(n) \\ \pi = (B_1, B_2, \ldots, B_k)}} \prod_{i=1}^{k-1} \kappa_{\abs{B_i}}(\mf{S}) \Cum{\underbrace{\mf{S}, \ldots, \mf{S}}_{\abs{B_k} \text{ times}}, \mf{V}, \mf{V}}\\
& = \alpha \beta \State{\mf{S}^n} + \alpha \beta \sum_{\substack{\pi \in \Int(n) \\ \pi = (B_1, B_2, \ldots, B_k)}} \prod_{i=1}^{k-1} \kappa_{\abs{B_i}}(\mf{S}) \kappa_{\abs{B_k} + 2}(\mf{S})
= \alpha \beta \State{\mf{S}^{n+2}}.
\end{split}
\]
Thus
\[
\State{\mf{S}^{n+2}} = \frac{1}{1 + b} \Bigl(\State{\mf{S}^n} + a \State{\mf{S}^{n+1}} + b \State{\mf{S}^{n+2}} \Bigr)
\]
and
\[
\sum_{n=0}^\infty \State{\mf{S}^{n+2}} z^{n+2} = \sum_{n=0}^\infty \State{\mf{S}^n} z^{n+2} + a \sum_{n=0}^\infty \State{\mf{S}^{n+1}} z^{n+2}
\]
Noting that $\State{\mf{S}} = 0$,
\[
M_{\mf{S}}(z) = z^2 (1 + M_{\mf{S}}(z)) + a zM_{\mf{S}}(z),
\]
so
\[
M_{\mf{S}}(z) (1 - a z - z^2) = z^2
\]
and
\[
M_{\mf{S}}(z) = \frac{z^2}{1 - a z - z^2}.
\]
Therefore
\[
G_{\mf{S}}(z) = \frac{1}{z} \left( 1 + M_{\mf{S}} \Bigl( \frac{1}{z} \Bigr) \right) = \frac{1}{z} \frac{1 - \frac{a}{z}}{1 - \frac{a}{z} - \frac{1}{z^2}} = \frac{z - a}{z^2 - a z - 1}.
\]
So the distribution of $\mf{S}$, and from \eqref{X-Y-S} also the distributions of $X$ and $Y$, are Bernoulli distributions \eqref{Binomial}, with $a$ replacing $b$.
\end{proof}

\subsection{Boolean Meixner process}
\label{Subsec:Process}

A construction inspired by \cite{Sniady-SWN} produces an infinite dimensional Boolean Meixner (or, in the language of \cite{Lytvynov-Meixner}, Jacobi) field. Let $\mc{A}$ be a complex $\ast$-algebra, $\mc{A}^{sa}$ its self-adjoint part, and $\psi$ a state on it, so that we can define the Hilbert space $\mc{H} = L^2(\mc{A}, \psi)$ obtained via the GNS construction. We identify
\[
L^2(\mc{A}, \psi) \otimes \ldots \otimes L^2(\mc{A}, \psi) \simeq L^2(\mc{A} \times \ldots \times \mc{A}, \psi \otimes \ldots \otimes \psi)
\]
with $L^2(\mc{A}, \psi)$ via the multiplication map
\[
\mk{m}: \mc{A} \times \ldots \times \mc{A} \rightarrow \mc{A},
\]
in other words we complete $\mc{A} \times \ldots \times \mc{A}$ with respect to the inner product
\[
\ip{f_1 \otimes f_2 \otimes \ldots \otimes f_n}{g_1 \otimes g_2 \otimes \ldots \otimes g_n} = \psi\left[ f_n \ldots f_1 g_1 \ldots g_n \right].
\]
With this identification, the full Fock space of $\mc{H}$
\[
\mf{C} \Omega \oplus \mc{H} \oplus \mc{H}^{\otimes 2} \oplus \mc{H}^{\otimes 3} \oplus \ldots
\]
collapses to the extended Boolean Fock space
\[
\mc{F}(\mc{H}) = \mf{C} \Omega \oplus \mc{H} \oplus \mc{H} \oplus \mc{H} \oplus \ldots,
\]
whose elements are of the form
\[
(\alpha, g_1, g_2, \ldots).
\]
For $f \in \mc{A}^{sa}$, define operators on $\mc{F}(\mc{H})$
\begin{align*}
a^+(f) (\alpha, g_1, g_2, \ldots) & = (0, \alpha f, f g_1, f g_2, \ldots), \\
a^0(f) (\alpha, g_1, g_2, \ldots) & = (0, f g_1, f g_2, \ldots), \\
a^-(f) (\alpha, g_1, g_2, \ldots) & = (\psi\left[f g_1 \right], 0, 0, \ldots), \\
\tilde{a}(f) (\alpha, g_1, g_2, \ldots) & = (0, f g_2, f g_3, f g_4, \ldots), \\
\end{align*}
and
\[
X(f) = a^+(f) + b a^0(f) + a^-(f) + c \tilde{a}(f),
\]
where $b \in \mf{R}$, $ c \geq 0$. On $\mc{F}(\mc{H})$, put the inner product
\[
\ip{(\alpha, g_1, g_2, \ldots)}{(\beta, h_1, h_2, \ldots)} = \bar{\alpha} \beta + \sum_{k=1}^\infty c^{k-1} \psi \left[g_k^\ast h_k \right].
\]
This is the same inner product as considered in \cite{Boz-Wys}, compare also with Section~5 of \cite{Krystek-Yoshida-t}.

\begin{Prop}
\br
\begin{enumerate}
\item
$\norm{a^+(f)} \leq \max(1, \sqrt{c}) \norm{f}_\infty$, $\norm{a^0(f)}_\infty \leq \norm{f}_\infty$, $a^0(f)$ is self-adjoint, and $(a^-(f) + \tilde{a}(f))$ is the adjoint of $a^+(f)$, so that for $\norm{f}_{\infty} < \infty$, $X(f)$ is bounded and self-adjoint.
\item
If $\set{f_1, f_2, \ldots, f_n}$ are pairwise orthogonal, meaning $f_i f_j = 0$ for $i \neq j$, then
\[
\set{X(f_1), X(f_2), \ldots, X(f_n)}
\]
are Boolean independent with respect to the vacuum state $\ip{\Omega}{\cdot \Omega}$. In fact, the joint Boolean cumulants of $(X(f_1), X(f_2), \ldots, X(f_n))$ are the same as the joint free cumulants of Theorem~8 of \cite{Sniady-SWN}, so this is yet another implementation of the BP bijection.
\item
For $(\mc{A}, \psi) = (L^\infty([0,1]), dx)$ and $f = \chf{[0,t)}$, the distribution of $X(f)$ is $\phi_{c, b}^{\uplus t}$.
\end{enumerate}
\end{Prop}

\begin{Remark}
If $c = 0$, which corresponds to the  Gaussian / Poisson case, this ``extended Boolean Fock space'' collapses to $\mf{C} \Omega \oplus \mc{H}$, which is a Boolean Fock space considered in \cite{Ben-Ghorbal-Boolean,Franz-Unification}.
\end{Remark}

\appendix

\section{Operator model, continued fractions, and applications}

\subsection{Multivariate continued fractions}

The purpose of this section is to prove Corollary~\ref{Cor:Boolean-semigroup}, which shows that the relation between recursion relations and the Boolean convolution used in Proposition~\ref{Prop:Free/Boolean} holds in full generality. In the one variable case, this relation was observed in \cite{Boz-Wys}, and follows easily from a continued fraction expansion for the Cauchy transform of the measure. We follows that proof, which requires an introduction of a multivariate continued fraction expansion for a moment generating function of a general state (with a MOPS; see below). This result may be of independent interest.

\begin{Remark}[General Fock space construction]
\label{Remark:General-Fock}
Let $\mc{H} = \mf{C}^d$, with the canonical orthonormal basis $e_1, \ldots, e_d$. Define the (algebraic) full Fock space of $\mc{H}$ to be
\[
\Falg(\mc{H}) = \mf{C} \Omega \oplus \bigoplus_{k=1}^\infty \mc{H}^{\otimes k}
\]

\br
For $i = 1, 2, \ldots, d$, define $a_i^+$ and $a_i^-$ to be the usual (left) free creation and annihilation operators,
\begin{align*}
a_i^+ & \left(e_{u(1)} \otimes e_{u(2)} \otimes \ldots \otimes e_{u(k)} \right) = e_i \otimes e_{u(1)} \otimes e_{u(2)} \otimes \ldots \otimes e_{u(k)}, \\
a_i^-      & (e_j) = \ip{e_i}{e_j} \Omega = \delta_{i j} \Omega, \\
a_i^-      & \left(e_{u(1)} \otimes e_{u(2)} \otimes \ldots \otimes e_{u(k)} \right) = \ip{e_i}{e_{u(1)}} e_{u(2)} \otimes \ldots \otimes e_{u(k)}.
\end{align*}
For future reference, we also define the Boolean annihilation operator
\[
a_i^{b-} =
\begin{cases}
a_i^- & \text{ on } \mc{H}^{\otimes k}, k \leq 1, \\
0, & \text{ on } \mc{H}^{\otimes k}, k \geq 2.
\end{cases}
\]

\br
For each $k \geq 1$ let $\mc{C}^{(k)}$ be a diagonal non-negative $d^k \times d^k$ matrix, with entries
\begin{equation}
\label{Expansion-C}
\mc{C}(e_{u(1)} \otimes \ldots \otimes e_{u(k)}) = \mc{C}_{\vec{u}} e_{u(1)} \otimes \ldots \otimes e_{u(k)}.
\end{equation}
Similarly, for each $i = 1, 2, \ldots, d$ and each $k \geq 0$, let $\mc{T}_i^{(k)}$ be a $d^k \times d^k$ matrix, with entries
\begin{equation}
\label{Expansion-T}
\mc{T}_i(e_{u(1)} \otimes \ldots \otimes e_{u(k)}) = \sum_{\abs{\vec{w}} = k} B_{i, \vec{w}, \vec{u}} e_{w(1)} \otimes \ldots \otimes e_{w(k)}.
\end{equation}
We identify $\mc{C}^{(k)}, \mc{T}_i^{(k)}$ with operators
\[
\mc{C}^{(k)}, \mc{T}_i^{(k)}: \mc{H}^{\otimes k} \rightarrow \mc{H}^{\otimes k}.
\]
Assume that $\mc{T}_i^{(k)}$ and $\mc{C}^{(j)}$ satisfy a commutation relation
\[
\left(\mc{T}_i^{(k)} \right)^\ast \mc{K}_{\mc{C}}^{(k)} = \mc{K}_{\mc{C}}^{(k)} \mc{T}_i^{(k)},
\]
where
\[
\mc{K}_{\mc{C}}^{(k)} = (I \otimes I \otimes \ldots \otimes I \otimes C^{(1)}) \ldots (I \otimes I \otimes C^{(k-2)}) (I \otimes C^{(k-1)}) C^{(k)}.
\]
We will denote by $\mc{T}_i$ and $\mc{C}$ the operators on $\Falg(\mc{H})$ acting as $\mc{T}_i^{(k)}$ and $\mc{C}^{(k)}$ on each component. Finally, let $\tilde{a}_i^- = a_i^- \mc{C}$ and
\[
\mc{X}_i = a_i^+ + \mc{T}_i + \tilde{a}_i^-.
\]
With the appropriate choice of the inner product $\ip{\cdot}{\cdot}_{\mc{C}}$ on the completion $\mc{F}_{\mc{C}}(\mc{H})$ of the quotient of $\Falg(\mc{H})$, all the operators $a_i^+, \mc{T}_i, \tilde{a}_i^-$ factor through to $\mc{F}_{\mc{C}}(\mc{H})$, and each $\mc{X}_i$ is a symmetric operator on it.
\end{Remark}

\begin{Prop}(Part of Theorem 2 of \cite{AnsMonic}, modified for complex-values states)
\label{Prop:Monic-states}
Let $\phi$ be a state on $\mf{C} \langle \mb{x} \rangle$. The following are equivalent:
\begin{enumerate}
\item
The state $\phi$ has a monic orthogonal polynomial system.
\item
For some choice of the matrices $\mc{C}^{(k)}$ and $\mc{T}_i^{(k)}$ as above, the state $\phi$ has a Fock space representation $\phi_{\mc{C}, \set{\mc{T}_i}}$ as
\begin{equation*}
\state{P(x_1, x_2, \ldots, x_d)} = \ip{\Omega}{P(\mc{X}_1, \mc{X}_2, \ldots, \mc{X}_d) \Omega}.
\end{equation*}
\item
There is a family of polynomials $\set{P_{\vec{u}}}$ such that $\state{P_{\vec{u}}} = 0$ for all $\vec{u} \neq \emptyset$ and they satisfy a recursion relation
\begin{align*}
x_i & = P_i + B_{i, \emptyset, \emptyset}, \\
x_i P_u & = P_{(i, u)} + \sum_{w=1}^d B_{i, w, u} P_{w} + \delta_{i, u} \mc{C}_u, \\
x_i P_{\vec{u}} & = P_{(i, \vec{u})} + \sum_{\abs{\vec{w}} = \abs{\vec{u}}} B_{i, \vec{w}, \vec{u}} P_{\vec{w}} + \delta_{i, u(1)} \mc{C}_{\vec{u}} P_{(u(2), u(3), \ldots, u(k))}.
\end{align*}
\end{enumerate}
\end{Prop}

\begin{Thm}
\label{Thm:Continued-fraction}
Let $\phi$ be a state with a MOPS, $\set{\mc{T}_i^{(k)}, \mc{C}^{(k)}}$ the matrices in its Fock space representation, whose entries are the coefficients in the recursion relation for the MOPS, and $M(\mb{z})$ its moment generating function. Then
\[
1 + M(\mb{z}) =
\cfrac{1}{1 - \sum_{i_0} z_{i_0} \mc{T}_{i_0}^{(0)} -
\cfrac{\sum_{j_1} z_{j_1} E_{j_1} \mc{C}^{(1)} | \sum_{k_1} E_{k_1} z_{k_1}}{1 - \sum_{i_1} z_{i_1} \mc{T}_{i_1}^{(1)} -
\cfrac{\sum_{j_2} z_{j_2} E_{j_2} \mc{C}^{(2)} | \sum_{k_2} E_{k_2} z_{k_2}}{1 - \sum_{i_2} z_{i_2} \mc{T}_{i_2}^{(2)} -
\cfrac{\sum_{j_3} z_{j_3} E_{j_3} \mc{C}^{(3)} | \sum_{k_3} E_{k_3} z_{k_3}}{1 - \ldots}}}}
\]
Here for matrices
\[
A, B \in M_{d^k \times d^k} \simeq M_{d \times d} \otimes M_{d \times d} \otimes \ldots \otimes M_{d \times d},
\]
we use the notation
\[
\frac{E_i A | E_j}{B} = \ip{e_i \otimes I \otimes \ldots \otimes I}{A B^{-1} (e_j \otimes I \otimes \ldots \otimes I)} \in M_{d^{k-1} \times d^{k-1}}.
\]
\end{Thm}

\begin{proof}
First treating
\[
\set{ \left. a_s = \left(\sum_{j_s} z_{j_s} E_{j_s} \mc{C}^{(s)}\right), b_s = \left(\sum_{k_s} E_{k_s} z_{k_s} \right), c_s = \left(\sum_{i_s} z_{i_s} \mc{T}_{i_s}^{(s)} \right) \right| \ s \in \mf{N}}
\]
simply as non-commuting symbols, we can apply the main Theorem~1 of Flajolet \cite{Flajolet} to represent the continued fraction above as a sum over Motzkin paths, with $\sum_{j_s} z_{j_s} E_{j_s} \mc{C}^{(s)}$ the weight of a falling step from level $s$ to level $s-1$, $\sum_{k_s} E_{k_s} z_{k_s}$ the weight of a rising step from level $s-1$ to level $s$, and $\sum_{i_s} z_{i_s} \mc{T}_{i_s}^{(s)}$ the weight of a horizontal step at level $s$. Then the coefficient of $z_{\vec{u}}$ in the expansion of this continued fraction is the sum over all Motzkin paths of length $\abs{\vec{u}}$, with the weight of the path $\rho$ equal to
\[
\prod_{i=1}^{\abs{\vec{u}}} \alpha_i,
\]
where
\[
\alpha_i =
\begin{cases}
E_{u(i)} \mc{C}^{(s)}, & \rho(i) \text{ is a falling step on level } s, \\
E_{u(i)}, & \rho(i) \text{ is a rising step on level } s, \\
\mc{T}_{u(i)}^{(s)}, & \rho(i) \text{ is a horizontal step on level } s.
\end{cases}
\]
But this product is equal to
\[
\ip{\Omega}{\prod_{i=1}^{\abs{\vec{u}}} \beta_i \Omega},
\]
where
\[
\alpha_i =
\begin{cases}
a_{u(i)}^- \mc{C}^{(s)}, & \rho(i) \text{ is a falling step on level } s, \\
a_{u(i)}^+, & \rho(i) \text{ is a rising step on level } s, \\
\mc{T}_{u(i)}^{(s)}, & \rho(i) \text{ is a horizontal step on level } s,
\end{cases}
\]
and the sum of such products over all Motzkin paths of length $\abs{\vec{u}}$ is exactly
\[
\ip{\Omega}{\mc{X}_{u(1)} \mc{X}_{u(2)} \ldots, \mc{X}_{u(\abs{\vec{u}})} \Omega} = \state{x_{\vec{u}}}. \qedhere
\]
\end{proof}

\begin{Ex}
If all $\mc{T}_i^{(k)} = 0$, the expansion takes a simpler (scalar) form
\[
1 + M(\mb{z}) =
\cfrac{1}{1 - \sum_{j_1} \mc{C}_{j_1}
\cfrac{z_{j_1} | z_{j_1}}{1 - \sum_{j_2} \mc{C}_{j_2 j_1}
\cfrac{z_{j_2} | z_{j_2}}{1 - \sum_{j_3} \mc{C}_{j_3 j_2 j_1}
\cfrac{z_{j_3} | z_{j_3}}{1 - \ldots}}}}
\]
In particular, in the ``simple quadratic'' free Meixner case $\mc{C}_{j_1} = 1$ and $\mc{C}_{j_k j_{k-1} \ldots j_1} = 1 + C_{j_k j_{k-1}}$, $C_{ij} = c$, so that
\[
1 + M(\mb{z}) =
\cfrac{1}{1 - \sum_{j_1}
\cfrac{z_{j_1} | z_{j_1}}{1 - \sum_{j_2} (1 + c)
\cfrac{z_{j_2} | z_{j_2}}{1 - \sum_{j_3} (1 + c)
\cfrac{z_{j_3} | z_{j_3}}{1 - \sum_{j_4} (1 + c)
\cfrac{z_{j_4} | z_{j_4}}{1 - \ldots}}}}}
\]
and
\[
1 - (1 + M)^{-1} = \sum_j z_j (1 + A) z_j,
\]
where
\[
1 + A(\mb{z}) = \cfrac{1}{1 - \sum_{j_2} (1 + c)
\cfrac{z_{j_2} | z_{j_2}}{1 - \sum_{j_3} (1 + c)
\cfrac{z_{j_3} | z_{j_3}}{1 - \sum_{j_4} (1 + c)
\cfrac{z_{j_4} | z_{j_4}}{1 - \ldots}}}}
\]
satisfies
\[
1 - (1 + A)^{-1} = (1 + c) \sum_j z_j (1 + A) z_j.
\]
Equivalently,
\[
M = \sum_j z_j (1 + A) z_j (1 + M)
\]
and
\[
A = (1 + c) \sum_j z_j (1 + A) z_j (1 + A)
\]
Note that in the one-dimensional case, the equations for $A$ and $M$ can be solved and give quadratic formulas. In the multivariate case of a free semicircular system, corresponding to $c = 0$, we get $M = A$ and so
\[
M = \sum_j z_j (1 + M) z_j (1 + M)
\]
which also follows directly from $R(\mb{z}) = \sum_j z_j^2$ and the transformation~\eqref{R-M-formula}.
\end{Ex}

\begin{Ex}
In the free product free Meixner case, $C_{ij} = \delta_{ij} c_i$, so
\[
1 + M(\mb{z}) =
\cfrac{1}
{1 -
\cfrac{\sum_{j_1} z_{j_1} | z_{j_1}}
{1 -
\cfrac{\sum_{j_2} z_{j_2} | z_{j_2}}
{1 -
\cfrac{\sum_{j_3} z_{j_3} | z_{j_3}}{1 - \ldots} -
\cfrac{c_{j_2} z_{j_2} | z_{j_2}}{1 - \ldots}
} -
\cfrac{c_{j_1} z_{j_1} | z_{j_1}}
{1 -
\cfrac{\sum_{j_2} z_{j_2} | z_{j_2}}{1 - \ldots} -
\cfrac{c_{j_1} z_{j_1} | z_{j_1}}{1 - \ldots}
}
}
}
\]
\end{Ex}

\begin{Ex}
For a general Boolean product
\[
1 - (1 + M)^{-1} = \eta= \sum_{j=1}^d \eta_j = \sum_{j=1}^d (1 - (1 + M_j)^{-1}).
\]
Thus using the notation from Example~\ref{Example:Product-states},
\[
1 + M(\mb{z}) =
\cfrac{1}{1 - \sum_{j} \beta_0^{(j)} z_{j} - \sum_j
\cfrac{\gamma^{(j)}_1 z_j^2}{1 - \beta^{(j)}_1 z_j -
\cfrac{\gamma^{(j)}_2 z_j^2}{1 - \beta^{(j)}_2 z_j -
\cfrac{\gamma^{(j)}_3 z_j^2}{1 - \ldots}}}}
\]
\end{Ex}

\subsection{General results in Boolean theory}

In this section, we collect a number of results which, while true for the Boolean Meixner class, hold in general. Typically their free or classical analogs, if any, hold only for the corresponding Meixner class.

\begin{Cor}
\label{Cor:Boolean-semigroup}
Let $\phi^{\uplus t}$ be the Boolean convolution power of a state $\phi$ with a MOPS. The coefficients in the recursion relation for the MOPS of $\phi^{\uplus t}$ are the same as for $\phi$, except that each $\mc{T}_i^{(0)}$ and $\mc{C}^{(1)}$ get multiplied by $t$.
\end{Cor}

\begin{proof}
From Theorem~\ref{Thm:Continued-fraction},
\[
\begin{split}
\eta_{\phi^{\uplus t}}
& = t \eta_{\phi}(\mb{z}) = t(1 - (1 + M_{\phi}(\mb{z}))^{-1}) \\
& = \sum_{i_0} z_{i_0} t \mc{T}_{i_0}^{(0)} +
\cfrac{\sum_{j_1} z_{j_1} E_{j_1} t \mc{C}^{(1)} | \sum_{k_1} E_{k_1} z_{k_1}}{1 - \sum_{i_1} z_{i_1} \mc{T}_{i_1}^{(1)} -
\cfrac{\sum_{j_2} z_{j_2} E_{j_2} \mc{C}^{(2)} | \sum_{k_2} E_{k_2} z_{k_2}}{1 - \sum_{i_2} z_{i_2} \mc{T}_{i_2}^{(2)} -
\cfrac{\sum_{j_3} z_{j_3} E_{j_3} \mc{C}^{(3)} | \sum_{k_3} E_{k_3} z_{k_3}}{1 - \ldots}}}. \qedhere
\end{split}
\]
\end{proof}

\br
The following relation between Boolean cumulants and Jacobi parameters was already noted by Lehner in the single variable case (relation 4.9 in \cite{Lehner-Cumulants-lattice}), if in rather different language.

\begin{Cor}
Let $\phi = \phi_{\mc{C}, \set{\mc{T}_i}}$ be a state with a MOPS. Using the terminology from Remark~\ref{Remark:General-Fock}, let
\[
\mc{Z}_i = \mc{X}_i - a_i^{b-} \mc{C}^{(1)} = a_i^+ + \mc{T}_i + \tilde{a}_i^{b-},
\]
where
\[
\tilde{a}_i^{b-} =
\begin{cases}
0 & \text{ on } \mc{H}^{\otimes k}, k \leq 1, \\
a_i^- \mc{C}, & \text{ on } \mc{H}^{\otimes k}, k \geq 2.
\end{cases}
\]
Then the Boolean cumulant functional $\eta$ of $\phi$ is $\Cum{x_i} = \mc{T}_0^{(i)}$,
\[
\Cum{x_i x_{\vec{u}} x_j} = \mc{C}_i \ip{e_i}{ \mc{Z}_{\vec{u}} e_j} = \ip{a_i^{b-} \Omega}{\mc{Z}_{\vec{u}} a_i^+ \Omega}.
\]
\end{Cor}

\br
Recall that for the free Meixner states, their free cumulant functional has a similar expression in terms of the operators
\[
S_i = X_i - a_i^- = a_i^+ + T_i + a_i^- (\mc{C} - I)
\]
but there is no such expansion for general states. Note also that the Boolean-to-free version of the Bercovici-Pata on the free Meixner states, discussed in Remark~\ref{Remark:BP-bijection}, follows from these observations.

\begin{Prop}
\label{Prop:General-Boolean-Fock}
Let $\mc{H}$ be a complex Hilbert space with a distinguished unit vector $\Omega$, so that $\mc{H} = \mf{C} \Omega \oplus \mc{H}_0$. Let $\set{X_1, X_2, \ldots, X_d}$ be symmetric operators on $\mc{H}$ with a common domain $\mc{D}$ such that $\Omega \in \mc{D}$ and for all $i$, $X_i(\mc{D}) \subset \mc{D}$. Let $\phi$ be the joint distribution of $\set{X_1, X_2, \ldots, X_d}$ with respect to $\Omega$, that is
\[
\state{P(x_1, x_2, \ldots, x_d)} = \ip{\Omega}{P(X_1, X_2, \ldots, X_d) \Omega}.
\]
Then there exist numbers $\lambda_i \in \mf{R}$, vectors $\xi_i \in \mc{H}_0$, and symmetric operators $T_i$ on $\mc{H}_0$ with domain $\mc{D}_0 = \mc{H}_0 \cap \mc{D}$, $T_i(\mc{D}_0) \subset \mc{D}_0$, such that each
\begin{equation}
\label{Operator-decomposition}
X_i = a_{\xi_i}^+ + a_{\xi_i}^- + T_i + \lambda_i I
\end{equation}
and the Boolean cumulant functional of $\phi$ is
\[
\Cum{x_i} = \lambda_i; \qquad \Cum{x_i x_{\vec{u}} x_j} = \ip{\xi_i}{T_{\vec{u}} \xi_j}.
\]
Here
\begin{align*}
a_{\xi_i}^+ \zeta &= \ip{\Omega}{\zeta} \xi_i, \\
a_{\xi_i}^- \zeta &= \ip{\xi_i}{\zeta} \Omega
\end{align*}
for $\zeta \in \mc{H}$ are rank one operators, which are clearly adjoints of each other.
\end{Prop}

\begin{proof}
Let $\lambda_i = \ip{\Omega}{X_i \Omega}$, $\xi_i = (X_i - \lambda_i I) \Omega$, and $T_i = X_i - a_{\xi_i}^+ - a_{\xi_i}^- - \lambda_i I$. Since $X_i$ is symmetric, $\lambda_i \in \mf{R}$ and $\ip{\xi_i}{\Omega} = 0$ so $\xi_i \in \mc{H}_0$. Clearly $\xi_i \in \mc{D}$, and since $X_i, a_{\xi_i}^+, a_{\xi_i}^-$ take $\mc{D}$ to itself, so does $T_i$. Moreover for $\zeta \in \mc{H}_0 \cap \mc{D}$,
\[
\ip{\Omega}{T_i \zeta}
= \ip{\Omega}{X_i \zeta - \ip{\xi_i}{\zeta} \Omega - \lambda_i \zeta}
= \ip{\Omega}{X_i \zeta} - \ip{X_i \Omega}{\zeta} + \lambda_i \ip{\Omega}{\zeta} - \lambda_i \ip{\Omega}{\zeta} = 0,
\]
so $T_i(\zeta) \in \mc{H}_0$. Since $a_{\xi_i}^+ + a_{\xi_i}^-$ is symmetric, so is $T_i$. It then easily follows that the Boolean cumulants of $\phi$ are $\Cum{x_i} = \lambda_i$ and
\[
\Cum{x_i x_{\vec{u}} x_j} = \ip{\Omega}{a_{\xi_i}^- T_{\vec{u}} a_{\xi_j}^+ \Omega} = \ip{\xi_i}{T_{\vec{u}} \xi_j}. \qedhere
\]
\end{proof}

\br
Note that there are bosonic and free versions of the operator decomposition~\eqref{Operator-decomposition} (see \cite{SchurCondPos} and \cite{GloSchurSpe}) but they only hold for operators with (freely) infinitely divisible joint distributions. The preceding proposition reflects the fact that all states are infinitely divisible in the Boolean sense \cite[Proposition 4.8]{Belinschi-Nica-Eta}.


\begin{thebibliography}{Ans07b}

\bibitem[Ans04]{AnsAppell}
Michael Anshelevich, \emph{Appell polynomials and their relatives}, Int. Math.
  Res. Not. (2004), no.~65, 3469--3531. \MR{MR2101359 (2005k:33012)}

\bibitem[Ans07a]{AnsFree-Meixner}
\bysame, \emph{Free {M}eixner states}, Commun. Math. Phys \textbf{276} (2007),
  no.~3, 863--899.

\bibitem[Ans08a]{AnsMonic}
\bysame, \emph{Monic non-commutative orthogonal polynomials},
  \texttt{arXiv:math/0702157 [math.CO]}, To be published by the
  Proceedings of the AMS, 2008.

\bibitem[Ans08b]{AnsMulti-Sheffer}
\bysame, \emph{Orthogonal polynomials with a resolvent-type generating
  function}, \texttt{arXiv:math/0410482 [math.CO]}, To be published by the
  Transactions of the AMS, 2008.

\bibitem[App80]{Appell}
M.~P. Appell, \emph{Sur une classe de polyn{\^o}mes}, Ann. Sci. Ecole Norm.
  Sup. \textbf{9} (1880), 119--144.

\bibitem[AS67]{Al-Salam-q-Appell}
Walled~A. Al-Salam, \emph{{$q$}-{A}ppell polynomials}, Ann. Mat. Pura Appl. (4)
  \textbf{77} (1967), 31--45. \MR{MR0223622 (36 \#6670)}

\bibitem[AT87]{Avram-Taqqu}
Florin Avram and Murad~S. Taqqu, \emph{Noncentral limit theorems and {A}ppell
  polynomials}, Ann. Probab. \textbf{15} (1987), no.~2, 767--775.
  \MR{88i:60058}

\bibitem[BB06]{Boz-Bryc}
Marek Bo{\.z}ejko and W{\l}odzimierz Bryc, \emph{On a class of free {L}\'evy
  laws related to a regression problem}, J. Funct. Anal. \textbf{236} (2006),
  no.~1, 59--77. \MR{MR2227129 (2007a:46071)}

\bibitem[Ber06]{Bercovici-Boolean}
Hari Bercovici, \emph{On {B}oolean convolutions}, Operator theory 20, Theta
  Ser. Adv. Math., vol.~6, Theta, Bucharest, 2006, pp.~7--13. \MR{MR2276927
  (2007m:46105)}

\bibitem[BGS02]{Ben-Ghorbal-Independence}
Anis Ben~Ghorbal and Michael Sch{\"u}rmann, \emph{Non-commutative notions of
  stochastic independence}, Math. Proc. Cambridge Philos. Soc. \textbf{133}
  (2002), no.~3, 531--561. \MR{MR1919720 (2003k:46096)}

\bibitem[BGS04]{Ben-Ghorbal-Boolean}
\bysame, \emph{Quantum stochastic calculus on {B}oolean {F}ock space}, Infin.
  Dimens. Anal. Quantum Probab. Relat. Top. \textbf{7} (2004), no.~4, 631--650.
  \MR{MR2105916 (2005j:81089)}

\bibitem[BLS96]{BLS96}
Marek Bo{\.z}ejko, Michael Leinert, and Roland Speicher, \emph{Convolution and
  limit theorems for conditionally free random variables}, Pacific J. Math.
  \textbf{175} (1996), no.~2, 357--388. \MR{MR1432836 (98j:46069)}

\bibitem[BN06]{Belinschi-Nica-Eta}
Serban~T. Belinschi and Alexandru Nica, \emph{$\eta$-series and a {B}oolean
  {B}ercovici-{P}ata bijection for bounded $k$-tuples},
  \texttt{arXiv:math/0608622 [math.OA]}, 2006.

\bibitem[BN07a]{Belinschi-Nica-B_t}
\bysame, \emph{On a remarkable semigroup of homomorphisms with respect to free
  multiplicative convolution}, \texttt{arXiv:math/0703295 [math.OA]}, 2007.

\bibitem[BN07b]{Belinschi-Nica-Free-BM}
\bysame, \emph{Free {B}rownian motion and evolution towards $\boxplus$-infinite
  divisibility for $k$-tuples}, \texttt{arXiv:0711.3787 [math.OA]}, 2007.

\bibitem[Bo{\.z}86]{Bozejko-Riesz-product}
Marek Bo{\.z}ejko, \emph{Positive definite functions on the free group and the
  noncommutative {R}iesz product}, Boll. Un. Mat. Ital. A (6) \textbf{5}
  (1986), no.~1, 13--21. \MR{MR833375 (88a:43007)}

\bibitem[Bo{\.z}87]{Bozejko-Free-groups}
\bysame, \emph{Uniformly bounded representations of free groups}, J. Reine
  Angew. Math. \textbf{377} (1987), 170--186. \MR{MR887407 (89a:22009)}

\bibitem[BP99]{BerPatDomains}
Hari Bercovici and Vittorino Pata, \emph{Stable laws and domains of attraction
  in free probability theory}, Ann. of Math. (2) \textbf{149} (1999), no.~3,
  1023--1060, With an appendix by Philippe Biane. \MR{2000i:46061}

\bibitem[BW01]{Boz-Wys}
Marek Bo{\.z}ejko and Janusz Wysocza{\'n}ski, \emph{Remarks on
  {$t$}-transformations of measures and convolutions}, Ann. Inst. H. Poincar\'e
  Probab. Statist. \textbf{37} (2001), no.~6, 737--761. \MR{MR1863276
  (2002i:60005)}

\bibitem[BW07]{Bryc-Wesolowski-Bi-Poisson}
W{\l}odzimierz Bryc and Jacek Weso{\l}owski, \emph{Bi-{P}oisson process},
  Infin. Dimens. Anal. Quantum Probab. Relat. Top. \textbf{10} (2007), no.~2,
  277--291. \MR{MR2337523}

\bibitem[Chi68]{Chihara-Brenke-1}
T.~S. Chihara, \emph{Orthogonal polynomials with {B}renke type generating
  functions}, Duke Math. J. \textbf{35} (1968), 505--517. \MR{MR0227488 (37
  \#3072)}

\bibitem[Fla80]{Flajolet}
P.~Flajolet, \emph{Combinatorial aspects of continued fractions}, Discrete
  Math. \textbf{32} (1980), no.~2, 125--161. \MR{MR592851 (82f:05002a)}

\bibitem[Fra03]{Franz-Unification}
Uwe Franz, \emph{Unification of {B}oolean, monotone, anti-monotone, and tensor
  independence and {L}\'evy processes}, Math. Z. \textbf{243} (2003), no.~4,
  779--816. \MR{MR1974583 (2004f:46077)}

\bibitem[GS86]{Giraitis-Surgailis}
L.~Giraitis and D.~Surgailis, \emph{Multivariate {A}ppell polynomials and the
  central limit theorem}, Dependence in probability and statistics
  (Oberwolfach, 1985), Progr. Probab. Statist., vol.~11, Birkh{\"a}user Boston,
  Boston, MA, 1986, pp.~21--71. \MR{89c:60024}

\bibitem[GSS92]{GloSchurSpe}
Peter Glockner, Michael Sch{\"u}rmann, and Roland Speicher, \emph{Realization
  of free white noises}, Arch. Math. (Basel) \textbf{58} (1992), no.~4,
  407--416. \MR{93e:46075}

\bibitem[KS05]{Kyprianou-Novikov-Shiryaev}
Andreas~E. Kyprianou and Budhi~A. Surya, \emph{On the {N}ovikov-{S}hiryaev
  optimal stopping problems in continuous time}, Electron. Comm. Probab.
  \textbf{10} (2005), 146--154 (electronic). \MR{MR2162814 (2006i:60051)}


\bibitem[KY04]{Krystek-Yoshida-t}
Anna Krystek and Hiroaki Yoshida, \emph{Generalized {$t$}-transformations of probability measures and
  deformed convolutions}, Probab. Math. Statist. \textbf{24} (2004), no.~1,
  Acta Univ. Wratislav. No. 2646, 97--119. \MR{MR2108159 (2006i:46093)}

\bibitem[Leh03]{Lehner-Cumulants-lattice}
Franz Lehner, \emph{Cumulants, lattice paths, and orthogonal polynomials},
  Discrete Math. \textbf{270} (2003), no.~1-3, 177--191. \MR{MR1997896
  (2005a:05224)}

\bibitem[Len05]{Lenczewski-Noncommutative-independence}
Romuald Lenczewski, \emph{On noncommutative independence}, Quantum probability
  and infinite dimensional analysis, QP--PQ: Quantum Probab. White Noise Anal.,
  vol.~18, World Sci. Publ., Hackensack, NJ, 2005, pp.~320--336. \MR{MR2212459
  (2007i:46064)}

\bibitem[LL60]{Laha-Lukacs}
R.~G. Laha and E.~Lukacs, \emph{On a problem connected with quadratic
  regression}, Biometrika \textbf{47} (1960), 335--343. \MR{MR0121922 (22
  \#12649)}

\bibitem[Lyt03]{Lytvynov-Meixner}
Eugene Lytvynov, \emph{Polynomials of {M}eixner's type in infinite
  dimensions---{J}acobi fields and orthogonality measures}, J. Funct. Anal.
  \textbf{200} (2003), no.~1, 118--149. \MR{MR1974091 (2004k:46063)}

\bibitem[Mei34]{Meixner}
J.~Meixner, \emph{Orthogonale polynomsysteme mit einer besonderen gestalt der
  erzeugenden funktion}, J. London Math. Soc. \textbf{9} (1934), 6--13.

\bibitem[M{\l}o04]{Mlotkowki-Lambda-Boolean}
Wojciech M{\l}otkowski, \emph{Limit theorems in {$\Lambda$}-{B}oolean
  probability}, Infin. Dimens. Anal. Quantum Probab. Relat. Top. \textbf{7}
  (2004), no.~3, 449--459. \MR{MR2085643 (2005e:46127)}

\bibitem[Mur03]{Muraki-Natural-products}
Naofumi Muraki, \emph{The five independences as natural products}, Infin.
  Dimens. Anal. Quantum Probab. Relat. Top. \textbf{6} (2003), no.~3, 337--371.
  \MR{MR2016316 (2005h:46093)}

\bibitem[NS06]{Nica-Speicher-book}
Alexandru Nica and Roland Speicher, \emph{Lectures on the combinatorics of free
  probability}, London Mathematical Society Lecture Note Series, vol. 335,
  Cambridge University Press, Cambridge, 2006. \MR{MR2266879}

\bibitem[Ora02]{Oravecz-Fermi}
Ferenc Oravecz, \emph{Fermi convolution}, Infin. Dimens. Anal. Quantum Probab.
  Relat. Top. \textbf{5} (2002), no.~2, 235--242. \MR{MR1914835 (2003e:46112)}


\bibitem[Pri01]{Privault-Boolean}
Nicolas Privault, \emph{Quantum stochastic calculus for the uniform measure and
  {B}oolean convolution}, S\'eminaire de Probabilit\'es, XXXV, Lecture Notes in
  Math., vol. 1755, Springer, Berlin, 2001, pp.~28--47. \MR{MR1837275
  (2002i:81146)}

\bibitem[Rom84]{Roman-Book}
Steven Roman, \emph{The umbral calculus}, Pure and Applied Mathematics, vol.
  111, Academic Press Inc. [Harcourt Brace Jovanovich Publishers], New York,
  1984. \MR{MR741185 (87c:05015)}

\bibitem[Rot75]{RotaFiniteCalculusBook}
Gian-Carlo Rota (ed.), \emph{Finite operator calculus}, Academic Press Inc.
  [Harcourt Brace Jovanovich Publishers], New York, 1975, With the
  collaboration of P. Doubilet, C. Greene, D. Kahaner, A. Odlyzko and R.
  Stanley. \MR{52 \#119}

\bibitem[Sch91]{SchurCondPos}
Michael Sch{\"u}rmann, \emph{Quantum stochastic processes with independent
  additive increments}, J. Multivariate Anal. \textbf{38} (1991), no.~1,
  15--35. \MR{92k:46113}

\bibitem[Sch00]{SchOrthogonal}
Wim Schoutens, \emph{Stochastic processes and orthogonal polynomials},
  Springer-Verlag, New York, 2000. \MR{2001f:60095}

\bibitem[{\'S}ni00]{Sniady-SWN}
Piotr {\'S}niady, \emph{Quadratic bosonic and free white noises}, Comm. Math.
  Phys. \textbf{211} (2000), no.~3, 615--628. \MR{MR1773810 (2001i:81148)}

\bibitem[Spe97]{SpeUniv}
Roland Speicher, \emph{On universal products}, Free probability theory
  (Waterloo, ON, 1995), Fields Inst. Commun., vol.~12, Amer. Math. Soc.,
  Providence, RI, 1997, pp.~257--266. \MR{MR1426844 (98c:46141)}

\bibitem[Sto05]{Stoica-Boolean}
George Stoica, \emph{Limit laws for normed and weighted {B}oolean
  convolutions}, J. Math. Anal. Appl. \textbf{309} (2005), no.~1, 369--374.
  \MR{MR2154049 (2006f:46060)}

\bibitem[SW97]{SW97}
Roland Speicher and Reza Woroudi, \emph{Boolean convolution}, Free probability
  theory (Waterloo, ON, 1995), Fields Inst. Commun., vol.~12, Amer. Math. Soc.,
  Providence, RI, 1997, pp.~267--279. \MR{MR1426845 (98b:46084)}

\bibitem[Tho45]{Thorne-Appell-sets}
C.~J. Thorne, \emph{A property of {A}ppell sets}, Amer. Math. Monthly
  \textbf{52} (1945), 191--193. \MR{MR0011753 (6,212b)}

\bibitem[Voi98]{Voi-Entropy5}
Dan Voiculescu, \emph{The analogues of entropy and of {F}isher's information
  measure in free probability theory. {V}. {N}oncommutative {H}ilbert
  transforms}, Invent. Math. \textbf{132} (1998), no.~1, 189--227.
  \MR{99d:46087}

\bibitem[Voi00]{VoiCoalgebra}
\bysame, \emph{The coalgebra of the free difference quotient and free
  probability}, Internat. Math. Res. Notices (2000), no.~2, 79--106.
  \MR{2001d:46096}

\bibitem[VS93]{Verde-Star}
Luis Verde-Star, \emph{Polynomial sequences of interpolatory type}, Stud. Appl.
  Math. \textbf{88} (1993), no.~3, 153--172. \MR{MR1204869 (94d:41008)}

\bibitem[vW73]{vWa73}
Wilhelm von Waldenfels, \emph{An approach to the theory of pressure broadening
  of spectral lines}, Probability and information theory, II, Springer, Berlin,
  1973, pp.~19--69. Lecture Notes in Math., Vol. 296.

\end{thebibliography}

\providecommand{\bysame}{\leavevmode\hbox to3em{\hrulefill}\thinspace}
\providecommand{\MR}{\relax\ifhmode\unskip\space\fi MR }
\providecommand{\MRhref}[2]{%
  \href{http://www.ams.org/mathscinet-getitem?mr=#1}{#2}
}
\providecommand{\href}[2]{#2}

\end{document}